\documentclass[11pt]{amsart}
\usepackage[letterpaper,margin=1in]{geometry}
\usepackage[T1]{fontenc}
\usepackage{lmodern}
\usepackage[utf8]{inputenc}
\usepackage{amsmath,amssymb,amsthm}
\usepackage{microtype}
\usepackage{tikz}
\usepackage[hidelinks]{hyperref}
\hypersetup{
  pdftitle={Spectral synthesis, Fourier extension, and arithmetic in finite groups},
  pdfauthor={Alex Iosevich, Zhihe Li, Ke Yu},
  pdfkeywords={spectral synthesis, Fourier extension, finite groups, paraboloids}
}
\theoremstyle{plain}
\newtheorem{theorem}{Theorem}[section]
\newtheorem{proposition}[theorem]{Proposition}
\newtheorem{lemma}[theorem]{Lemma}
\newtheorem{corollary}[theorem]{Corollary}
\theoremstyle{definition}
\newtheorem{definition}[theorem]{Definition}
\theoremstyle{remark}

\numberwithin{equation}{section}

\title[Spectral synthesis and Fourier extension]{Spectral synthesis, Fourier extension, and arithmetic in ${\mathbb Z}_N^d$}
\author{Alex Iosevich}
\address{Department of Mathematics, University of Rochester, Rochester, NY, USA}
\email{iosevich@gmail.com}
\thanks{Alex Iosevich was supported in part by National Science Foundation grant DMS-2506858 and a 2026 Simons Fellowship in Mathematics.}

\author{Zhihe Li}
\address{Department of Mathematics, University of Rochester, Rochester, NY, USA}
\email{zli154@ur.rochester.edu}

\author{Ke Yu}
\address{Department of Mathematics, University of Rochester, Rochester, NY, USA}
\email{kyu26@ur.rochester.edu}

\date{}

\begin{document}
\begin{abstract}
We ask when a uniform bound for the Fourier $\ell^p$ norm of a function supported on a small subset of ${\mathbb Z}_N^d$ forces its largest Fourier coefficient to tend to zero as $N$ grows, a property we call spectral synthesis. Our main tool is the Fourier ratio, the quotient of the $\ell^1$ and $\ell^2$ norms of the Fourier transform. If the support has cardinality $N^{\alpha+o(1)}$, where $0<\alpha<d$, and the Fourier ratio is at most $N^{\frac{d}{2}-\kappa+o(1)}$, we prove synthesis for $2\leq p<\frac{2(d-2\kappa)}{\alpha-2\kappa}$ when $0\leq\kappa<\frac{\alpha}{2}$, and for every finite $p\geq2$ when $\kappa=\frac{\alpha}{2}$. A bound for the Fourier ratio of the indicator of the support leads to the same conclusion for every function supported there. Random unions of cosets show that this range cannot be enlarged for any pair $(\alpha,\kappa)$ in a prescribed dimension, along powers of a fixed odd prime. Using the random selection estimates of Burstein, Iosevich and Krause, we also find supports of cardinality $\lfloor N^\alpha\rfloor$ on which every nonzero function has Fourier ratio comparable to $N^{\frac{d}{2}}$ and synthesis holds exactly for $1\leq p<\frac{2d}{\alpha}$.

For sets with a transitive group of affine symmetries we determine the exact synthesis constant. At every even exponent its reciprocal is the norm of the Fourier extension operator with equal input and output exponents, taken with the probability measure on the support and counting measure on the output, and interpolation then yields further $L^p$ to $\ell^q$ estimates. For paraboloids modulo odd integers these constants are explicit divisor sums. They distinguish prime moduli, powers of a fixed prime and products of distinct primes, even though all three families have the same Fourier ratio up to a factor $N^{o(1)}$. They also tell us when synthesis holds at an endpoint and when the extension estimates at even exponents have constants independent of the modulus. Over odd squarefree moduli the sharp $L^2$ to $\ell^4$ extension norm of the parabola is again the reciprocal of the synthesis constant at $4$, and we explain how such exact identities differ from estimates that allow a factor $N^\varepsilon$ for every $\varepsilon>0$. Finally, for $1<p\leq2$ we determine the range of $L^p$ to $\ell^q$ extension estimates for paraboloids over prime powers in every dimension, and over odd squarefree moduli when a factor $N^\varepsilon$ is allowed.
\end{abstract}
\maketitle

\section{Introduction}\label{sec:introduction}

Let us begin with a simple question. Suppose that a function on ${\mathbb Z}_N^d$ is supported on a small set and that its Fourier transform has bounded $\ell^p$ norm. Must its largest Fourier coefficient become small as $N$ grows? When the answer is yes for every such sequence of functions, we say that the supports admit spectral synthesis at the exponent $p$, and Definition \ref{def:uniform-synthesis} states this formally. The answer depends on $p$, and a first application of the Cauchy--Schwarz inequality already tells us where to look.

Write ${\mathbb Z}_N={\mathbb Z}/N{\mathbb Z}$, where $N\geq2$, and let $d$ be a positive integer. For a function $f:{\mathbb Z}_N^d\to{\mathbb C}$ we use the unitary Fourier transform
$$
\widehat f(\xi)=N^{-\frac{d}{2}}\sum_{x\in{\mathbb Z}_N^d}f(x)e^{-\frac{2\pi ix\cdot\xi}{N}},
$$
where $x\cdot\xi=\sum_{j=1}^dx_j\xi_j$. The $\ell^r$ norms use counting measure, so that $\|f\|_r=(\sum_x|f(x)|^r)^{\frac{1}{r}}$ for $1\leq r<\infty$ and $\|f\|_\infty=\max_x|f(x)|$. The support of $f$, denoted $\operatorname{supp}(f)$, is the set where $f$ is nonzero. If it is contained in a set $S\subseteq{\mathbb Z}_N^d$, then the triangle inequality, the Cauchy--Schwarz inequality and the Plancherel identity tell us that
\begin{equation}\label{eq:intro-support}
\|\widehat f\|_\infty\leq N^{-\frac{d}{2}}\|f\|_1\leq N^{-\frac{d}{2}}|S|^{\frac{1}{2}}\|f\|_2=N^{-\frac{d}{2}}|S|^{\frac{1}{2}}\|\widehat f\|_2.
\end{equation}
For $p\geq2$, H\"older's inequality over the $N^d$ frequencies yields $\|\widehat f\|_2\leq N^{d(\frac{1}{2}-\frac{1}{p})}\|\widehat f\|_p$, and therefore
$$
\|\widehat f\|_\infty\leq|S|^{\frac{1}{2}}N^{-\frac{d}{p}}\|\widehat f\|_p.
$$
We write $o(1)$ for a quantity tending to zero as $N$ grows, and we say that sets $S_N$ have cardinality exponent $\alpha$ if $|S_N|=N^{\alpha+o(1)}$. For such sets the last estimate settles the question for $p<\frac{2d}{\alpha}$, and the issue is whether further information about $S_N$ or $f_N$ can enlarge this range.

A hyperplane and a paraboloid over an odd prime modulus make a useful comparison. Both contain $N^{d-1}$ points, but the Fourier transform of the indicator of the hyperplane lives on the line orthogonal to it, while that of the paraboloid is nonzero at almost every frequency. Cardinality cannot see this difference, so any improvement must use information beyond the size of the support.

For a nonzero function $f$ we define its Fourier ratio by
$$
\operatorname{FR}(f)=\frac{\|\widehat f\|_1}{\|\widehat f\|_2}.
$$
It lies between $1$ and $N^{\frac{d}{2}}$, and it is small when the Fourier transform is concentrated, in the sense measured by these two norms. For a sequence $\{f_N\}$ of nonzero functions, the power of $N$ gained over the largest possible ratio is recorded by the concentration exponent
\begin{equation}\label{eq:intro-kappa}
\kappa(\{f_N\})=\liminf_{N\to\infty}\frac{\log\left(\frac{N^{\frac{d}{2}}}{\operatorname{FR}(f_N)}\right)}{\log N}.
\end{equation}
When each $f_N$ is supported on $S_N$ and the sets $S_N$ have cardinality exponent $\alpha$, Fourier inversion forces $0\leq\kappa(\{f_N\})\leq\frac{\alpha}{2}$, as we check in Lemma \ref{lem:support-fr}, and larger values of $\kappa$ correspond to greater concentration. We call a pair $(\alpha,\kappa)$ admissible if $0<\alpha<d$ and $0\leq\kappa\leq\frac{\alpha}{2}$, and we speak of maximal concentration when $\kappa=\frac{\alpha}{2}$.

Our question is a finite analogue of a familiar uncertainty principle. Agranovsky and Narayanan \cite{AgranovskyNarayanan} proved that an $L^p({\mathbb R}^d)$ function whose Fourier transform is supported on an $\alpha$-dimensional $C^1$ submanifold, where $1\leq\alpha<d$ is an integer, must vanish when $2\leq p\leq\frac{2d}{\alpha}$. Deodhar and Iosevich \cite{DeodharIosevich} refined this with the Fourier ratio in the Euclidean and manifold settings, and the interpolation used below is the finite group form of their estimate. In a single finite group a uniqueness statement of this kind is impossible without further conditions, so we ask instead for a quantitative estimate whose constant tends to zero along a sequence of groups.

Bhowmik, Deodhar and Iosevich \cite{BhowmikDeodharIosevich} established the basic finite group estimates and showed that they are sharp, including the cardinality estimate above and the convolution estimate used below. Their formulation places the support condition on the Fourier transform, and the unitary Fourier transform interchanges it with ours. Our aim is to measure how much the Fourier ratio adds to these estimates. We show that the resulting range cannot be improved for any admissible pair, and we use affine symmetry to compute exact synthesis constants for arithmetic sets.

We write $1_S$ for the indicator of $S$, equal to one on $S$ and zero elsewhere. The Fourier ratio can enter in two ways. We may impose a condition on the particular function under study, or we may impose a condition on the indicator of a set and seek an estimate for every function supported there. The following theorem handles both situations.

\begin{theorem}\label{thm:intro-finite}
Let $\varnothing\neq S\subseteq{\mathbb Z}_N^d$, and let $f\neq0$ be supported on $S$. For $2\leq p<\infty$,
\begin{equation}\label{eq:intro-finite}
\|\widehat f\|_\infty\leq N^{-\frac{d}{2}}|S|^{\frac{1}{2}}\min\{\operatorname{FR}(f),\operatorname{FR}(1_S)\}^{1-\frac{2}{p}}\|\widehat f\|_p.
\end{equation}
\end{theorem}

We prove the theorem in Section \ref{sec:synthesis}. The first bound comes from using the Fourier ratio of $f$ to improve the comparison between the $\ell^2$ and $\ell^p$ norms. The second starts from the identity $f=f1_S$ and uses convolution on the Fourier side, and this is the argument that allows arbitrary weights on $S$.

Suppose now that $S_N$ has cardinality exponent $\alpha$, where $0<\alpha<d$, and that each $f_N$ is supported on $S_N$. Theorem \ref{thm:intro-finite} implies that
\begin{equation}\label{eq:intro-critical}
\sup_N\|\widehat f_N\|_p<\infty\text{ implies }\|\widehat f_N\|_\infty\longrightarrow0
\end{equation}
whenever
$$
2\leq p<p_c(\alpha,\kappa)=\frac{2(d-2\kappa)}{\alpha-2\kappa},
$$
where $\kappa<\frac{\alpha}{2}$ may be the concentration exponent of $\{f_N\}$ or that of $\{1_{S_N}\}$, and the larger of the two may be used when both are known. Under maximal concentration every finite $p\geq2$ is allowed.

This range is universal, and no statement that depends only on $\alpha$ and $\kappa$ can improve it. In Theorem \ref{thm:all-pair-sharpness} we construct endpoint counterexamples for every $0<\alpha<d$ and $0\leq\kappa<\frac{\alpha}{2}$ in the prescribed dimension $d$, along powers of a fixed odd prime. A subgroup supplies the Fourier concentration, a random collection of its cosets supplies the rest of the support, and a single moment estimate controls both the critical norm and the Fourier ratio.

The Fourier ratio of a function supported on $S$ can be much smaller than that of $1_S$, for instance when $S$ contains a large subgroup. Are there sets on which every nonzero function has a large Fourier ratio? Random sets are the obvious candidates, and the tool that controls every function on them at once is the restricted isometry property. A matrix has this property if it nearly preserves the squared Euclidean norm of every vector with a prescribed bound on the number of its nonzero coordinates. Estimates of this kind for randomly sampled Fourier matrices go back to Cand\`es and Tao \cite{CandesTao} and to Rudelson and Vershynin \cite{RudelsonVershynin}, and we use the recent improvement by Burstein, Iosevich and Krause \cite{BIKRIP} of the estimate of Haviv and Regev \cite{HavivRegev}.

Read in the dual direction, the restricted isometry estimate bounds the Fourier energy that a function supported on a random set can place on any sufficiently small set of frequencies, simultaneously for all such functions. Proposition \ref{prop:random-supports} turns this into a lower bound for the Fourier ratio that misses $N^{\frac{d}{2}}$ by only a logarithm. For supports of cardinality $\lfloor N^\alpha\rfloor$, with $0<\alpha<d$ fixed, the logarithm disappears if we use instead the selection theorem of Bourgain \cite{Bourgain}, in the quantitative form of Burstein, Iosevich and Krause \cite{BIKLambda}. Corollary \ref{cor:random-supports} then shows that every function on such a support has Fourier ratio comparable to $N^{\frac{d}{2}}$, and that synthesis holds exactly in the range dictated by cardinality.

A particular family can still have a larger synthesis range, and to see how this happens we introduce, for a nonempty set $S$ and $1\leq r<\infty$, the synthesis constant
$$
\mathcal C_r(S)=\sup_{\substack{f\neq0\\ \operatorname{supp}(f)\subseteq S}}\frac{\|\widehat f\|_\infty}{\|\widehat f\|_r}.
$$
An invertible affine map has the form $x\mapsto Ax+b$, where $A$ is an invertible matrix over ${\mathbb Z}_N$, and a group of such maps acts transitively on $S$ if it preserves $S$ and can take any point of $S$ to any other. Averaging over such a group, we prove in Theorem \ref{thm:affine-symmetry} that
\begin{equation}\label{eq:intro-symmetry}
\mathcal C_r(S)=\frac{N^{-\frac{d}{2}}|S|}{\|\widehat{1_S}\|_r}.
\end{equation}
In other words, the indicator alone determines the optimal constant for every choice of weights on the set. The proof moves the largest Fourier coefficient to the origin and then averages over the group.

Synthesis is also tied exactly to Fourier extension. We give $S$ the probability measure $\mu_S$ that assigns mass $\frac{1}{|S|}$ to each point, and for a function $g$ on $S$ we define the extension operator
\begin{equation}\label{eq:intro-extension}
E_Sg(\xi)=\frac{1}{|S|}\sum_{x\in S}g(x)e^{\frac{2\pi ix\cdot\xi}{N}}.
\end{equation}
For $1\leq p<\infty$ let $\|g\|_{L^p(\mu_S)}=(\frac{1}{|S|}\sum_{x\in S}|g(x)|^p)^{\frac{1}{p}}$, and let $\|g\|_{L^\infty(\mu_S)}=\max_{x\in S}|g(x)|$. The best extension constant is
$$
A_{p,q}(S)=\sup_{g\neq0}\frac{\|E_Sg\|_q}{\|g\|_{L^p(\mu_S)}},
$$
where $1\leq p,q\leq\infty$ and the output norm still uses counting measure. A uniform extension estimate for a family $\{S_N\}$ is a bound for $A_{p,q}(S_N)$ that does not depend on $N$.

\begin{theorem}\label{thm:intro-extension}
Suppose that a group of invertible affine maps acts transitively on a nonempty set $S\subseteq{\mathbb Z}_N^d$. For every even integer $r\geq2$ and every $r\leq p\leq\infty$,
\begin{equation}\label{eq:intro-reciprocal}
A_{p,r}(S)=\|E_S1\|_r=\frac{1}{\mathcal C_r(S)}.
\end{equation}
\end{theorem}

We prove Theorem \ref{thm:intro-extension} in Section \ref{sec:extension}. To see why an even exponent helps, write $r=2m$, expand $\|E_Sg\|_{2m}^{2m}$ and sum over the frequency first. The exponential sum vanishes except on tuples of points of $S$ satisfying
$$
x_1+\cdots+x_m=y_1+\cdots+y_m.
$$
Affine maps preserve this equation, so by transitivity the number of such tuples with a prescribed value of any one coordinate does not depend on that value, and H\"older's inequality now delivers the upper bound in \eqref{eq:intro-reciprocal}. The function $g=1$ attains it. Interpolation with the trivial $L^1(\mu_S)$ to $\ell^\infty$ bound then produces estimates with different input and output exponents, and nothing in the argument requires the modulus to be prime or odd.

Our main arithmetic example is the paraboloid
$$
P_N^{(d)}=\{(t_1,\ldots,t_{d-1},t_1^2+\cdots+t_{d-1}^2):t\in{\mathbb Z}_N^{d-1}\},
$$
where $N$ is odd and $d\geq2$. Let $\varphi(M)$ denote the number of residues modulo $M$ that are relatively prime to $M$, with $\varphi(1)=1$. Combining the formula \eqref{eq:intro-symmetry} with an evaluation of Gauss sums, we show in Theorem \ref{thm:paraboloid-exact} that
\begin{equation}\label{eq:intro-divisor}
\mathcal C_r(P_N^{(d)})=\left(\sum_{M\mid N}\varphi(M)M^{(d-1)(1-\frac{r}{2})}\right)^{-\frac{1}{r}}.
\end{equation}
These Gauss sums and divisor sums are classical, and they occur in the restriction analysis of Hickman and Wright \cite{HickmanWright}. What is new is their role as exact synthesis constants, and at every even $r\geq2$ Theorem \ref{thm:intro-extension} identifies the corresponding extension norm as the $r$-th root of the same divisor sum.

Already in dimension two, different sequences of moduli behave differently. Along primes, synthesis holds for $1\leq r<4$ and fails for $r\geq4$. Along powers of a fixed odd prime it holds for $1\leq r\leq4$, with endpoint constant comparable to $(\log N)^{-\frac{1}{4}}$, and along products of all odd primes up to a growing bound it holds for $1\leq r\leq6$. All three families have cardinality exponent $1$ and concentration exponent $0$. The Fourier ratio thus provides the right universal estimate, but the full collection of Fourier moments retains arithmetic information that the ratio cannot see.

The reason is visible in the divisor sum. For powers of a fixed prime it is a geometric series, while for products of distinct primes the contributions of the different primes multiply, and this accounts for the jump from $4$ to $6$. In fact, Corollary \ref{cor:parabola-six} shows that synthesis at $6$ holds along a family of odd moduli exactly when the sum of the reciprocals of the distinct prime divisors tends to infinity.

Figure \ref{fig:arithmetic-ranges} places the three families side by side. The size of the support and the concentration exponent are the same in each row, and only the arithmetic of the modulus moves the endpoint.

\begin{figure}[htbp]
\centering
\begin{tikzpicture}[x=1cm,y=1cm,font=\small]
\node at (3,3.25) {$d=2$ and $(\alpha,\kappa)=(1,0)$};
\draw[gray!50,densely dashed,line width=0.4pt] (3,-0.4)--(3,2.85);
\draw[gray!50,densely dashed,line width=0.4pt] (5,-0.4)--(5,2.85);
\node[anchor=east,align=right] at (-0.3,2.5) {Odd primes};
\node[anchor=east,align=right] at (-0.3,1.4) {Powers of a fixed\\odd prime};
\node[anchor=east,align=right] at (-0.3,0.3) {Products of odd primes\\up to $Y$};
\draw[->,gray!55,line width=0.4pt] (0,2.5)--(6.25,2.5);
\draw[->,gray!55,line width=0.4pt] (0,1.4)--(6.25,1.4);
\draw[->,gray!55,line width=0.4pt] (0,0.3)--(6.25,0.3);
\draw[line width=1.8pt] (0,2.5)--(3,2.5);
\draw[line width=1.8pt] (0,1.4)--(3,1.4);
\draw[line width=1.8pt] (0,0.3)--(5,0.3);
\fill (0,2.5) circle[radius=2.6pt];
\fill (0,1.4) circle[radius=2.6pt];
\fill (0,0.3) circle[radius=2.6pt];
\draw[fill=white,line width=1pt] (3,2.5) circle[radius=3pt];
\fill (3,1.4) circle[radius=3pt];
\fill (5,0.3) circle[radius=3pt];
\node[anchor=west] at (6.55,2.5) {$1\leq r<4$};
\node[anchor=west] at (6.55,1.4) {$1\leq r\leq4$};
\node[anchor=west] at (6.55,0.3) {$1\leq r\leq6$};
\draw[->,line width=0.5pt] (0,-0.4)--(6.25,-0.4);
\draw (0,-0.34)--(0,-0.46) node[below=2pt] {$1$};
\draw (1,-0.34)--(1,-0.46) node[below=2pt] {$2$};
\draw (2,-0.34)--(2,-0.46) node[below=2pt] {$3$};
\draw (3,-0.34)--(3,-0.46) node[below=2pt] {$4$};
\draw (4,-0.34)--(4,-0.46) node[below=2pt] {$5$};
\draw (5,-0.34)--(5,-0.46) node[below=2pt] {$6$};
\draw (6,-0.34)--(6,-0.46) node[below=2pt] {$7$};
\node[anchor=west] at (6.4,-0.4) {$r$};
\end{tikzpicture}
\caption{Exact synthesis ranges for the parabola $P_N^{(2)}$. Black intervals indicate synthesis, and filled and open circles indicate included and excluded endpoints. In the third row $N$ is the product of all odd primes at most $Y$, with $Y\to\infty$. Every row has $(\alpha,\kappa)=(1,0)$. See the formula \eqref{eq:prime-exact}, Corollary \ref{cor:prime-power-endpoint} and Corollary \ref{cor:primorial}.}
\label{fig:arithmetic-ranges}
\end{figure}

The reciprocal identity also separates synthesis from extension. For an even exponent $r$, synthesis on a family of transitive sets means that $\mathcal C_r(S_N)\to0$, or equivalently that $A_{r,r}(S_N)\to\infty$, whereas a uniform diagonal extension estimate, that is, a bound for $A_{r,r}(S_N)$ independent of $N$, means that $\inf_N\mathcal C_r(S_N)>0$. Neither alternative needs to hold, since a family can contain one subsequence on which the extension norms stay bounded and another on which they tend to infinity.

The dependence on the modulus becomes important when we turn to the $L^2$ extension problem. A modulus is squarefree if it is not divisible by the square of any prime. For squarefree moduli, Kingsbury-Neuschotz \cite{KingsburyNeuschotz} bounds the $L^2$ to $\ell^4$ extension norm of the parabola by a factor that depends on the number of prime divisors, and in Proposition \ref{prop:squarefree-sharp-extension} we compute this norm exactly for odd squarefree moduli and find that it equals $\mathcal C_4(P_N^{(2)})^{-1}$. This norm can tend to infinity while staying below $C_\varepsilon N^\varepsilon$ for every $\varepsilon>0$, where $C_\varepsilon$ depends only on $\varepsilon$. In the comparison with synthesis, then, the constants matter as much as the exponents.

Over powers of a prime the comparison extends to every dimension. Proposition \ref{prop:prime-power-extension-range} determines the extension range for $1<p\leq2$, and the necessity of its condition comes from a density supported on a subgroup contained in the paraboloid. For squarefree moduli, Corollary \ref{cor:squarefree-extension-range} determines the corresponding range for $1<p\leq2$ when a factor $N^\varepsilon$ is allowed, using known estimates over prime fields, and in dimension three a theorem of Gonz\'alez-Riquelme and Oliveira e Silva \cite{GonzalezOliveiraSilva} leads to an exact $L^2$ to $\ell^4$ identity for odd squarefree moduli. The full squarefree problem in higher dimensions remains open, and in Section \ref{subsec:higher-dimensions} we explain why the linear subspaces contained in the paraboloid have to enter its formulation.

The paper is organized as follows. Section \ref{sec:notation} fixes the normalizations, Section \ref{sec:synthesis} proves the synthesis estimates and Section \ref{sec:symmetry} the affine averaging formula. The geometric and arithmetic examples in Sections \ref{sec:geometry} and \ref{sec:arithmetic} show what the formula detects beyond the exponents $\alpha$ and $\kappa$. Section \ref{sec:random} contains the random constructions, Section \ref{sec:structure} relates concentration to Fourier moments and additive energy, and Section \ref{sec:extension} proves the extension theorem and develops its arithmetic consequences. We close with a list of open questions.

\section{Fourier normalization and concentration parameters}\label{sec:notation}

Throughout the paper $N\geq2$ and ${\mathbb Z}_N={\mathbb Z}/N{\mathbb Z}$. A family of sets or functions may be indexed by any unbounded set of moduli, and all limits are then taken within that set, a convention we shall use for primes, prime powers and products of distinct primes. Functions at different moduli need not be related to one another, so synthesis means decay in norm, and not that a function eventually vanishes identically.

For $f:{\mathbb Z}_N^d\to{\mathbb C}$ we define
$$
\widehat f(\xi)=N^{-\frac{d}{2}}\sum_{x\in{\mathbb Z}_N^d}f(x)e^{-\frac{2\pi ix\cdot\xi}{N}},
$$
and the inversion and Plancherel formulas read
$$
f(x)=N^{-\frac{d}{2}}\sum_{\xi\in{\mathbb Z}_N^d}\widehat f(\xi)e^{\frac{2\pi ix\cdot\xi}{N}}\text{ and }\|f\|_2=\|\widehat f\|_2.
$$
A character of an abelian group $G$ is a function $\chi:G\to{\mathbb C}$ with $|\chi(x)|=1$ and $\chi(x+y)=\chi(x)\chi(y)$, and it is nontrivial if it is not identically one. The exponential functions above are the characters of ${\mathbb Z}_N^d$, and both identities follow from the orthogonality relation
\begin{equation}\label{eq:character-orthogonality}
\sum_{\xi\in{\mathbb Z}_N^d}e^{\frac{2\pi ix\cdot\xi}{N}}=N^d1_{\{0\}}(x),
\end{equation}
where $1_{\{0\}}$ is the indicator of the origin. To verify it, factor the sum into $d$ geometric series. When $x\neq0$, at least one of them has ratio different from one, and that series sums to zero. Unless a measure is displayed, norms use counting measure, and for $1\leq p\leq\infty$ we write $p'$ for the conjugate exponent, determined by $\frac{1}{p}+\frac{1}{p'}=1$, with $p'=\infty$ when $p=1$ and $p'=1$ when $p=\infty$. Convolution is unnormalized, $(F*G)(\xi)=\sum_\eta F(\eta)G(\xi-\eta)$, so that
\begin{equation}\label{eq:product-transform}
\widehat{fg}=N^{-\frac{d}{2}}\widehat f*\widehat g.
\end{equation}
The notation $A\lesssim B$, or $B\gtrsim A$, means that $A\leq CB$ for a constant $C$ independent of the modulus. Fixed dimensions and exponents may enter $C$, and a subscript records any further dependence, so that $C_\varepsilon$, for instance, depends on $\varepsilon$. We write $A\asymp B$ when both comparisons hold and $A=O(B)$ when $|A|\lesssim B$. As in the introduction, $o(1)$ denotes a quantity tending to zero, and $N^{o(1)}$ denotes a positive factor whose logarithm divided by $\log N$ tends to zero. We call such a factor subpolynomial.

\begin{definition}\label{def:exponents}
For $f\neq0$ we set
$$
\operatorname{FR}(f)=\frac{\|\widehat f\|_1}{\|\widehat f\|_2}\text{ and }\kappa_N(f)=\frac{\log\left(\frac{N^{\frac{d}{2}}}{\operatorname{FR}(f)}\right)}{\log N},
$$
and we call $\kappa_N(f)$ the concentration exponent of $f$ at the modulus $N$. For a sequence of nonzero functions the lower and upper concentration exponents are
$$
\kappa(\{f_N\})=\liminf_{N\to\infty}\kappa_N(f_N)\text{ and }\overline\kappa(\{f_N\})=\limsup_{N\to\infty}\kappa_N(f_N).
$$
A sequence of nonempty sets has cardinality exponent $\alpha$ if $\frac{\log|S_N|}{\log N}\to\alpha$, or equivalently if $|S_N|=N^{\alpha+o(1)}$.
\end{definition}

The lower exponent is the right parameter for an estimate that must eventually hold at every modulus. The upper exponent will enter the obstruction to extension in Section \ref{sec:extension}, because a uniform extension estimate must also hold along the moduli where concentration is greatest. When the exponents at each modulus converge, the two agree.

\begin{lemma}\label{lem:support-fr}
If $f\neq0$ is supported on $S\subseteq{\mathbb Z}_N^d$, then
$$
\frac{N^{\frac{d}{2}}}{|S|^{\frac{1}{2}}}\leq\operatorname{FR}(f)\leq N^{\frac{d}{2}}.
$$
Consequently, if the sets $S_N$ have cardinality exponent $\alpha$ and $f_N$ is supported on $S_N$, then
$$
0\leq\kappa(\{f_N\})\leq\overline\kappa(\{f_N\})\leq\frac{\alpha}{2}.
$$
\end{lemma}

\begin{proof}
The upper bound is the Cauchy--Schwarz inequality on the Fourier side, $\|\widehat f\|_1\leq N^{\frac{d}{2}}\|\widehat f\|_2$. For the lower bound we use Fourier inversion and the support condition,
$$
\|\widehat f\|_1\geq N^{\frac{d}{2}}\|f\|_\infty\text{ and }\|\widehat f\|_2=\|f\|_2\leq|S|^{\frac{1}{2}}\|f\|_\infty,
$$
and divide the first inequality by the second. Taking logarithms, we get $0\leq\kappa_N(f_N)\leq\frac{\log|S_N|}{2\log N}$, and the assertions about the limits follow.
\end{proof}

For indicators it is convenient to keep one more normalization in view. The Fourier algebra norm of $S$ is
\begin{equation}\label{eq:algebra-norm}
\mathcal A_N(S)=N^{-\frac{d}{2}}\|\widehat{1_S}\|_1,
\end{equation}
the sum of the absolute values of the Fourier coefficients of $1_S$ when the transform is normalized by $N^{-d}$ instead of $N^{-\frac{d}{2}}$. Fourier inversion at a point of $S$ shows that $\mathcal A_N(S)\geq1$ for nonempty $S$, and the definitions lead directly to
\begin{equation}\label{eq:algebra-kappa}
\operatorname{FR}(1_S)=\frac{N^{\frac{d}{2}}\mathcal A_N(S)}{|S|^{\frac{1}{2}}}\text{ and }\kappa_N(1_S)=\frac{\log|S|}{2\log N}-\frac{\log\mathcal A_N(S)}{\log N}.
\end{equation}
For sets of cardinality exponent $\alpha$ it follows that
\begin{equation}\label{eq:maximal-algebra}
\kappa(\{1_{S_N}\})=\frac{\alpha}{2}\text{ if and only if }\mathcal A_N(S_N)=N^{o(1)}.
\end{equation}
Indeed, by \eqref{eq:algebra-kappa} the left side of \eqref{eq:maximal-algebra} says that the nonnegative sequence $\frac{\log\mathcal A_N(S_N)}{\log N}$ has upper limit zero, that is, that it tends to zero, and this is the right side.

\section{The synthesis estimate}\label{sec:synthesis}

\begin{definition}\label{def:uniform-synthesis}
Let $1\leq r<\infty$. A family of nonempty sets $\{S_N\}$ admits uniform spectral synthesis at exponent $r$ if every sequence of functions $f_N$ supported on $S_N$ satisfies
$$
\sup_N\|\widehat f_N\|_r<\infty\text{ implies }\|\widehat f_N\|_\infty\longrightarrow0.
$$
\end{definition}

The proof of Theorem \ref{thm:intro-finite} has two parts. The first uses the Fourier ratio of $f$ itself. The second uses only the indicator of $S$, and so it bounds every function supported on $S$ at once.

\begin{proof}[Proof of Theorem \ref{thm:intro-finite}]
For $p=2$ the assertion is just \eqref{eq:intro-support}, so let $p>2$. Writing $|\widehat f|^2=|\widehat f|^{2\theta}|\widehat f|^{2(1-\theta)}$ and applying H\"older's inequality, which amounts to the logarithmic convexity of the $\ell^r$ norms, we get
$$
\|\widehat f\|_2\leq\|\widehat f\|_1^\theta\|\widehat f\|_p^{1-\theta},
$$
where $\theta$ is determined by $\frac{1}{2}=\theta+\frac{1-\theta}{p}$. Thus $\theta=\frac{p-2}{2(p-1)}$ and $\frac{\theta}{1-\theta}=1-\frac{2}{p}$. We now substitute $\|\widehat f\|_1=\operatorname{FR}(f)\|\widehat f\|_2$, divide by $\|\widehat f\|_2^\theta$ and raise both sides to the power $\frac{1}{1-\theta}$. The result is
\begin{equation}\label{eq:interpolation-fr}
\|\widehat f\|_2\leq\operatorname{FR}(f)^{1-\frac{2}{p}}\|\widehat f\|_p,
\end{equation}
and together with \eqref{eq:intro-support} it proves the bound involving $\operatorname{FR}(f)$.

For a bound that depends only on $S$, we insert the identity $f=f1_S$ into \eqref{eq:product-transform} and apply H\"older's inequality to the convolution, arriving at
$$
\|\widehat f\|_\infty\leq N^{-\frac{d}{2}}\|\widehat{1_S}\|_{p'}\|\widehat f\|_p.
$$
This is the estimate involving the indicator from \cite[Theorem 5]{BhowmikDeodharIosevich}, written in our formulation. Since $\frac{1}{p'}=(1-\frac{2}{p})+\frac{1}{2}\cdot\frac{2}{p}$, the same convexity, now between $\ell^1$ and $\ell^2$, yields
$$
\|\widehat{1_S}\|_{p'}\leq\|\widehat{1_S}\|_1^{1-\frac{2}{p}}\|\widehat{1_S}\|_2^{\frac{2}{p}}=\operatorname{FR}(1_S)^{1-\frac{2}{p}}|S|^{\frac{1}{2}},
$$
where we used $\|\widehat{1_S}\|_1=\operatorname{FR}(1_S)|S|^{\frac{1}{2}}$ and $\|\widehat{1_S}\|_2=|S|^{\frac{1}{2}}$. This proves the second bound, and the theorem follows by taking the smaller of the two right-hand sides.
\end{proof}

In terms of the Fourier algebra norm \eqref{eq:algebra-norm}, the estimate involving the indicator reads
\begin{equation}\label{eq:algebra-synthesis}
\|\widehat f\|_\infty\leq\left(\frac{|S|}{N^d}\right)^{\frac{1}{p}}\mathcal A_N(S)^{1-\frac{2}{p}}\|\widehat f\|_p.
\end{equation}
The first factor is small when $S$ occupies a small proportion of the group, and the second measures the extra information carried by the Fourier transform of its indicator.

\begin{theorem}\label{thm:asymptotic}
Suppose that the sets $S_N$ have cardinality exponent $\alpha$, where $0<\alpha<d$, and that $f_N\neq0$ is supported on $S_N$. Set $\kappa_f=\kappa(\{f_N\})$ and $\kappa_S=\kappa(\{1_{S_N}\})$, and let $\kappa=\max\{\kappa_f,\kappa_S\}$. For every fixed $2\leq p<\infty$ and $\varepsilon>0$, all sufficiently large $N$ satisfy
\begin{equation}\label{eq:asymptotic-estimate}
\|\widehat f_N\|_\infty\leq N^{\frac{\alpha-d}{2}+(\frac{d}{2}-\kappa)(1-\frac{2}{p})+\varepsilon}\|\widehat f_N\|_p.
\end{equation}
If $\kappa<\frac{\alpha}{2}$, bounded $\ell^p$ norms therefore imply $\|\widehat f_N\|_\infty\to0$ whenever
$$
2\leq p<\frac{2(d-2\kappa)}{\alpha-2\kappa},
$$
and if $\kappa=\frac{\alpha}{2}$ this conclusion holds for every finite $p\geq2$. Each conclusion also holds with either $\kappa_f$ or $\kappa_S$ alone, and in the latter case it is uniform over all functions supported on $S_N$.
\end{theorem}

\begin{proof}
Fix $\delta>0$. For all large $N$ we have $|S_N|\leq N^{\alpha+\delta}$, and by the definition of the lower limits, after discarding finitely many more moduli,
$$
\operatorname{FR}(f_N)\leq N^{\frac{d}{2}-\kappa_f+\delta}\text{ and }\operatorname{FR}(1_{S_N})\leq N^{\frac{d}{2}-\kappa_S+\delta}.
$$
The smaller of the two ratios is therefore at most $N^{\frac{d}{2}-\kappa+\delta}$, and Theorem \ref{thm:intro-finite} yields \eqref{eq:asymptotic-estimate} once $\delta$ is small in terms of $\varepsilon$. Without $\varepsilon$, the exponent in \eqref{eq:asymptotic-estimate} equals
$$
\frac{\alpha-d}{2}+\left(\frac{d}{2}-\kappa\right)\left(1-\frac{2}{p}\right)=\frac{\alpha-2\kappa}{2}-\frac{d-2\kappa}{p}.
$$
When $\kappa<\frac{\alpha}{2}$ it is negative exactly in the stated range, and at maximal concentration it equals $-\frac{d-\alpha}{p}$, which is negative for every finite $p$. Choosing $\varepsilon$ smaller than its absolute value, we conclude that bounded $\ell^p$ norms force $\|\widehat f_N\|_\infty\to0$.
\end{proof}

For $\kappa=0$ we recover the range $p<\frac{2d}{\alpha}$, and for $0<\kappa<\frac{\alpha}{2}$ the identity
$$
\frac{2(d-2\kappa)}{\alpha-2\kappa}-\frac{2d}{\alpha}=\frac{4\kappa(d-\alpha)}{\alpha(\alpha-2\kappa)}>0
$$
tells us that the gain is strict. For $1\leq p<2$ there is nothing to prove, because $\|\widehat f_N\|_2\leq\|\widehat f_N\|_p$, and the estimate \eqref{eq:intro-support} already implies synthesis when $\alpha<d$.

Let us look at the conclusion when $d=2$ and $p=4$. The condition in Theorem \ref{thm:asymptotic} becomes $\alpha<1+\kappa$, the shaded region in Figure \ref{fig:parameter-region}, and greater concentration at a fixed support size moves the parameters up toward this region. The parabola sits on its boundary, and its behavior at the endpoint depends on the arithmetic information displayed in Figure \ref{fig:arithmetic-ranges}.

\begin{figure}[htbp]
\centering
\begin{tikzpicture}[x=4.6cm,y=4.6cm,font=\small]
\fill[gray!20] (0,0)--(1,0)--(2,1)--cycle;
\draw[->,line width=0.6pt] (0,0)--(2.17,0) node[right] {$\alpha$};
\draw[->,line width=0.6pt] (0,0)--(0,1.13) node[above] {$\kappa$};
\draw[gray!60,densely dotted,line width=0.5pt] (2,0)--(2,1);
\draw[line width=0.9pt] (0,0)--(2,1);
\draw[dashed,line width=0.9pt] (1,0)--(2,1);
\draw (1,0.018)--(1,-0.018) node[below=3pt] {$1$};
\draw (2,0.018)--(2,-0.018) node[below=3pt] {$2$};
\draw (0.018,0.5)--(-0.018,0.5) node[left=3pt] {$\frac{1}{2}$};
\draw (0.018,1)--(-0.018,1) node[left=3pt] {$1$};
\node[below left=3pt] at (0,0) {$0$};
\node[align=center] at (0.92,0.26) {Synthesis\\guaranteed};
\node[text width=2.7cm,align=center] at (1.6,0.16) {No universal\\conclusion};
\node[rotate=26.565,anchor=south,inner sep=5pt] at (1.5,0.75) {$\kappa=\frac{\alpha}{2}$};
\node[rotate=45,anchor=north,inner sep=5pt] at (1.63,0.63) {$\kappa=\alpha-1$};
\fill (1,0.5) circle[radius=2.6pt];
\node[anchor=south east,inner sep=4pt] at (1,0.5) {$H$};
\draw[fill=white,line width=0.9pt] (1,0) circle[radius=2.8pt];
\node[anchor=south east,inner sep=5pt] at (1,0) {$P$};
\draw[fill=white,line width=0.7pt] (0,0) circle[radius=2pt];
\draw[fill=white,line width=0.7pt] (2,0) circle[radius=2pt];
\draw[fill=white,line width=0.7pt] (2,1) circle[radius=2pt];
\node[anchor=north east] at (2.1,1.13) {$d=2$ and $p=4$};
\end{tikzpicture}
\caption{The universal synthesis region for $d=2$ and $p=4$. Admissible parameters satisfy $0<\alpha<2$ and $0\leq\kappa\leq\frac{\alpha}{2}$, with maximal concentration on the upper edge. Theorem \ref{thm:asymptotic} applies in the shaded region $\alpha<1+\kappa$, and the dashed critical line is excluded from it. The hyperplane has parameters $H=(1,\frac{1}{2})$, and every parabola family in Figure \ref{fig:arithmetic-ranges} has parameters $P=(1,0)$. The two exponents alone cannot decide what happens at $P$.}
\label{fig:parameter-region}
\end{figure}

The estimate at each modulus says more than the two lower limits do separately, and $\kappa$ in Theorem \ref{thm:asymptotic} can be replaced by
$$
\liminf_{N\to\infty}\max\{\kappa_N(f_N),\kappa_N(1_{S_N})\}.
$$
This helps when the two sources of concentration are effective at different moduli, but we will not need the refinement.

At a finite critical exponent the powers of $N$ cancel. The estimate of Theorem \ref{thm:intro-finite} still implies synthesis at the endpoint $p_c$ if
$$
N^{-\frac{d}{2}}|S_N|^{\frac{1}{2}}\operatorname{FR}(1_{S_N})^{1-\frac{2}{p_c}}\longrightarrow0,
$$
a condition that can detect logarithmic improvements. It is not necessary, as the arithmetic examples of Section \ref{sec:arithmetic} will show.

\section{Exact synthesis constants from affine symmetry}\label{sec:symmetry}

What is the best constant in the synthesis inequality? Recall that for nonempty $S\subseteq{\mathbb Z}_N^d$ and $1\leq r<\infty$ the synthesis constant is
\begin{equation}\label{eq:synthesis-constant}
\mathcal C_r(S)=\sup_{\substack{f\neq0\\ \operatorname{supp}(f)\subseteq S}}\frac{\|\widehat f\|_\infty}{\|\widehat f\|_r}.
\end{equation}
Uniform synthesis in the sense of Definition \ref{def:uniform-synthesis} is equivalent to $\mathcal C_r(S_N)\to0$. If the constants tend to zero, multiplying them by bounded input norms produces the required decay. If they do not, there are $c>0$ and a subsequence on which $\mathcal C_r(S_N)\geq c$, and choosing there functions with Fourier $\ell^r$ norm one and largest Fourier coefficient at least $\frac{c}{2}$, and zero functions at the other moduli, we obtain a sequence that violates uniform synthesis.

At the exponent $2$ the answer depends only on the cardinality. For every nonempty $S$
$$
\mathcal C_2(S)=\left(\frac{|S|}{N^d}\right)^{\frac{1}{2}},
$$
since \eqref{eq:intro-support} provides the upper bound and the indicator attains it. Moreover, $\mathcal C_r(S)$ is nondecreasing in $r$, because $\ell^r$ norms with respect to counting measure decrease as $r$ grows, so synthesis at one exponent implies synthesis at every smaller exponent.

Let $T(x)=Ax+b$ be an invertible affine map, so that the matrix $A$ is invertible over ${\mathbb Z}_N$. The substitution $y=Ax+b$ in the definition of the Fourier transform leads to
\begin{equation}\label{eq:affine-transform}
\widehat{f\circ T}(\xi)=e^{\frac{2\pi i(A^{-1}b)\cdot\xi}{N}}\widehat f(A^{-T}\xi),
\end{equation}
where $A^{-T}$ denotes the transpose of $A^{-1}$, and composition with $T$ therefore preserves every Fourier norm that we use. Multiplying $f(x)$ by $e^{-\frac{2\pi ix\cdot\xi_0}{N}}$ is called modulation. It translates the Fourier transform by $\xi_0$, so it too preserves these norms.

\begin{theorem}\label{thm:affine-symmetry}
Let $\Gamma$ be a group of invertible affine maps of ${\mathbb Z}_N^d$ that preserves a nonempty set $S$ and acts transitively on it. Then, for $1\leq r<\infty$,
\begin{equation}\label{eq:exact-symmetry}
\mathcal C_r(S)=\frac{N^{-\frac{d}{2}}|S|}{\|\widehat{1_S}\|_r}.
\end{equation}
For $1<r<\infty$, the nonzero functions attaining equality are exactly the scalar multiples of modulated indicators,
$$
f(x)=ce^{\frac{2\pi ix\cdot\xi_0}{N}}1_S(x),
$$
where $c\neq0$ and $\xi_0\in{\mathbb Z}_N^d$.
\end{theorem}

\begin{proof}
Fix $f$ supported on $S$ and a frequency $\xi_0$. Modulating $f$, we obtain a function $g$, still supported on $S$, with $\widehat g(0)=\widehat f(\xi_0)$, and we average it over the finite group $\Gamma$ by setting
$$
h(x)=\frac{1}{|\Gamma|}\sum_{T\in\Gamma}g(Tx).
$$
Every term vanishes when $x\notin S$. When $x\in S$, each point of $S$ occurs equally often among the points $Tx$, because if $T_0y=z$, then $T\mapsto T_0T$ is a bijection from the maps taking $x$ to $y$ onto the maps taking $x$ to $z$. Hence
$$
h=\frac{1}{|S|}\left(\sum_{x\in S}g(x)\right)1_S=\frac{N^{\frac{d}{2}}}{|S|}\widehat f(\xi_0)1_S,
$$
and by \eqref{eq:affine-transform} and the triangle inequality,
$$
\frac{N^{\frac{d}{2}}}{|S|}|\widehat f(\xi_0)|\|\widehat{1_S}\|_r=\|\widehat h\|_r\leq\frac{1}{|\Gamma|}\sum_{T\in\Gamma}\|\widehat{g\circ T}\|_r=\|\widehat f\|_r.
$$
Taking the maximum over $\xi_0$ proves the upper bound in \eqref{eq:exact-symmetry}. The indicator attains it, because its Fourier coefficients have absolute value at most $\widehat{1_S}(0)=N^{-\frac{d}{2}}|S|$.

Now let $1<r<\infty$, suppose that $f\neq0$ attains equality, and choose $\xi_0$ where $|\widehat f|$ is largest, so that $\widehat f(\xi_0)\neq0$. Every inequality in the argument above is then an equality, and in particular the average of the vectors $\widehat{g\circ T}$, all of which have norm $\|\widehat f\|_r$, has this same norm. The space $\ell^r$ is strictly convex for $1<r<\infty$, which means that equality in the triangle inequality for finitely many nonzero vectors forces them to be positive multiples of one another. Our vectors have equal norms, so they coincide. Since the identity map belongs to $\Gamma$, Fourier inversion tells us that $g\circ T=g$ for every $T\in\Gamma$, and transitivity makes $g$ constant on $S$. Undoing the modulation proves the assertion. Conversely, every function of the displayed form attains the same ratio as $1_S$.
\end{proof}

The proof has a simple interpretation. After modulation, the chosen Fourier coefficient is the sum of the weights, and averaging makes the weights constant on $S$ without increasing the Fourier norm. This is why the indicator is extremal even among complex valued weights.

\begin{lemma}\label{lem:products}
For nonzero functions $f$ on ${\mathbb Z}_N^{d_1}$ and $g$ on ${\mathbb Z}_N^{d_2}$, let $(f\otimes g)(x,y)=f(x)g(y)$. Then
$$
\operatorname{FR}(f\otimes g)=\operatorname{FR}(f)\operatorname{FR}(g),
$$
so that the concentration exponents at each modulus add, and their lower limits satisfy
$$
\kappa(\{f_N\otimes g_N\})\geq\kappa(\{f_N\})+\kappa(\{g_N\}),
$$
with equality if both exponents at each modulus converge. For nonempty sets $S$ and $T$ and $1\leq r<\infty$,
\begin{equation}\label{eq:product-constants}
\mathcal C_r(S\times T)=\mathcal C_r(S)\mathcal C_r(T).
\end{equation}
\end{lemma}

\begin{proof}
The Fourier transform of a tensor product is the tensor product of the Fourier transforms, so its $\ell^1$ and $\ell^2$ norms factor. This proves the identity for the Fourier ratio and, after taking logarithms, the statements about the exponents.

For the upper bound in \eqref{eq:product-constants}, let $F=\widehat f$ with $f$ supported on $S\times T$. For fixed $\eta$, the function $F(\cdot,\eta)$ is the Fourier transform, in the first variable, of the partial Fourier transform of $f$ in the second variable evaluated at $\eta$, and this function is supported on $S$. In the same way, $F(\xi,\cdot)$ is the Fourier transform of a function supported on $T$. Hence
$$
\sup_{\xi,\eta}|F(\xi,\eta)|\leq\mathcal C_r(S)\sup_\eta\left(\sum_\xi|F(\xi,\eta)|^r\right)^{\frac{1}{r}}\leq\mathcal C_r(S)\left(\sum_\xi\sup_\eta|F(\xi,\eta)|^r\right)^{\frac{1}{r}},
$$
and applying the bound in the second variable to each term of the last sum, we arrive at
$$
\sup_{\xi,\eta}|F(\xi,\eta)|\leq\mathcal C_r(S)\mathcal C_r(T)\left(\sum_{\xi,\eta}|F(\xi,\eta)|^r\right)^{\frac{1}{r}}.
$$
For the reverse inequality we take tensor products of maximizers for the two factors. Maximizers exist because the ratio does not change under scaling and the unit sphere of a finite dimensional space is compact.
\end{proof}

\section{Geometric examples over prime moduli}\label{sec:geometry}

We begin with examples in which the Fourier transform can be computed explicitly. Throughout this section $N$ runs through the odd primes unless we say otherwise.

\subsection{Hyperplanes and paraboloids}

For $d\geq2$ let
$$
H_N={\mathbb Z}_N^{d-1}\times\{0\}\text{ and }P_N^{(d)}=\{(t,|t|^2):t\in{\mathbb Z}_N^{d-1}\},
$$
where $|t|^2=t_1^2+\cdots+t_{d-1}^2$ is computed in ${\mathbb Z}_N$. Both sets have $N^{d-1}$ elements. By orthogonality,
$$
\widehat{1_{H_N}}(\xi)=N^{\frac{d}{2}-1}1_{\{\xi_1=\cdots=\xi_{d-1}=0\}}(\xi),
$$
so the indicator of the hyperplane has $N$ nonzero Fourier coefficients, all of the same magnitude. Consequently $\operatorname{FR}(1_{H_N})=N^{\frac{1}{2}}$ and $\kappa(\{1_{H_N}\})=\frac{d-1}{2}$, the largest possible value. Translations within the hyperplane act transitively, and Theorem \ref{thm:affine-symmetry} yields
\begin{equation}\label{eq:hyperplane-exact}
\mathcal C_r(H_N)=N^{-\frac{1}{r}}.
\end{equation}
None of these statements about $H_N$ uses the primality of $N$.

For the paraboloid, orthogonality shows that the Fourier transform of the indicator equals $N^{\frac{d}{2}-1}$ at the origin and vanishes at the other frequencies with $\xi_d=0$. When $\xi_d\neq0$ it factors into $d-1$ quadratic Gauss sums, that is, sums of the form $\sum_{t\bmod N}e^{-\frac{2\pi i}{N}(at^2+bt)}$ with $a\neq0$, and each of these has magnitude $N^{\frac{1}{2}}$. Therefore
\begin{equation}\label{eq:prime-paraboloid}
|\widehat{1_{P_N^{(d)}}}(\xi)|=N^{\frac{d}{2}-1}1_{\{0\}}(\xi)+N^{-\frac{1}{2}}1_{\{\xi_d\neq0\}}(\xi).
\end{equation}
Lemma \ref{lem:gauss} below computes the magnitude of these sums for every odd modulus. In particular, $\|\widehat{1_{P_N^{(d)}}}\|_1\asymp N^{d-\frac{1}{2}}$ and $\operatorname{FR}(1_{P_N^{(d)}})\asymp N^{\frac{d}{2}}$, so the concentration exponent of the paraboloid is zero.

For $b\in{\mathbb Z}_N^{d-1}$ the affine map
\begin{equation}\label{eq:paraboloid-symmetry}
T_b(x',x_d)=(x'+b,x_d+2b\cdot x'+|b|^2)
\end{equation}
takes $(t,|t|^2)$ to $(t+b,|t+b|^2)$. Since $T_b\circ T_c=T_{b+c}$, these maps form a transitive group of symmetries of the paraboloid, and their linear parts have determinant one, so they are invertible even when the modulus is composite. In the prime case the formula \eqref{eq:prime-paraboloid} tells us that $\|\widehat{1_{P_N^{(d)}}}\|_r^r=N^{r(\frac{d}{2}-1)}+(N-1)N^{d-1}N^{-\frac{r}{2}}$, and Theorem \ref{thm:affine-symmetry} yields
\begin{equation}\label{eq:prime-exact}
\mathcal C_r(P_N^{(d)})=\left(1+(N-1)N^{(d-1)(1-\frac{r}{2})}\right)^{-\frac{1}{r}}.
\end{equation}
The term after $1$ tends to infinity exactly when $r<\frac{2d}{d-1}$. At $r=\frac{2d}{d-1}$ the constant tends to $2^{-\frac{1}{r}}$, so synthesis fails at the critical exponent.

\subsection{A sphere}

Let
$$
\Sigma_N=\{x\in{\mathbb Z}_N^d:x_1^2+\cdots+x_d^2=1\},
$$
where $d\geq2$. Its Fourier ratio has the same power of $N$ as that of the paraboloid, although the computation is less explicit. Let $\chi$ be the quadratic character, defined by $\chi(0)=0$, $\chi(s)=1$ when $s$ is a nonzero square in ${\mathbb Z}_N$ and $\chi(s)=-1$ otherwise, and let $G=\sum_{t\in{\mathbb Z}_N}e^{\frac{2\pi it^2}{N}}$. We write
$$
1_{\Sigma_N}(x)=\frac{1}{N}\sum_{s\in{\mathbb Z}_N}e^{\frac{2\pi is(|x|^2-1)}{N}},
$$
complete the square in each coordinate, and use the identity $\sum_ye^{\frac{2\pi isy^2}{N}}=\chi(s)G$ for $s\neq0$. For $\xi\neq0$ the term $s=0$ disappears, and we find that
$$
\widehat{1_{\Sigma_N}}(\xi)=N^{-\frac{d}{2}-1}G^d\sum_{s\neq0}\chi(s)^de^{-\frac{2\pi i}{N}(s+\frac{|\xi|^2}{4s})},
$$
where $\frac{1}{4s}$ denotes the inverse of $4s$ modulo $N$. When $|\xi|^2\neq0$, the Weil bounds for Kloosterman and twisted Kloosterman sums show that the last sum has magnitude at most a constant times $N^{\frac{1}{2}}$, as explained in \cite{IosevichRudnev,IwaniecKowalski}. When $|\xi|^2=0$ it is an ordinary additive character sum if $d$ is even and a quadratic Gauss sum if $d$ is odd, so its magnitude is $1$ or $N^{\frac{1}{2}}$. Since $|G|=N^{\frac{1}{2}}$, we arrive at
$$
|\widehat{1_{\Sigma_N}}(\xi)|\lesssim N^{-\frac{1}{2}}\text{ for }\xi\neq0.
$$
At the origin, the same calculation produces
$$
|\Sigma_N|=N^{d-1}+\frac{G^d}{N}\sum_{s\neq0}\chi(s)^de^{-\frac{2\pi is}{N}}.
$$
The last sum equals $-1$ when $d$ is even and has magnitude $N^{\frac{1}{2}}$ when $d$ is odd, so $|\Sigma_N|=N^{d-1}+O(N^{\frac{d-1}{2}})$. In particular $|\Sigma_N|\asymp N^{d-1}$, and
$$
\sum_{\xi\neq0}|\widehat{1_{\Sigma_N}}(\xi)|^2=|\Sigma_N|-\frac{|\Sigma_N|^2}{N^d}\asymp N^{d-1}.
$$
If we bound one factor in this sum of squares by the largest nonzero coefficient, we find that $N^{d-1}\lesssim N^{-\frac{1}{2}}\|\widehat{1_{\Sigma_N}}\|_1$. The reverse estimate $\|\widehat{1_{\Sigma_N}}\|_1\lesssim N^{d-\frac{1}{2}}$ follows from the Cauchy--Schwarz inequality and Plancherel, and consequently
$$
\operatorname{FR}(1_{\Sigma_N})\asymp N^{\frac{d}{2}}\text{ and }\kappa(\{1_{\Sigma_N}\})=0.
$$

\subsection{Products and a geometric endpoint example}

Products mix these two kinds of Fourier behavior. If $d_1,d_2\geq2$, the set $\Sigma_N\times H_N\subseteq{\mathbb Z}_N^{d_1+d_2}$, with the sphere in ${\mathbb Z}_N^{d_1}$ and the hyperplane in ${\mathbb Z}_N^{d_2}$, has $\alpha=d_1+d_2-2$ and $\kappa=\frac{d_2-1}{2}$ by Lemma \ref{lem:products}, a value strictly between zero and $\frac{\alpha}{2}$.

Products also produce endpoint examples. Consider
$$
S_N={\mathbb Z}_N^m\times P_N^{(n)},
$$
where $d=m+n$, $m\geq0$ and $n\geq2$. The indicator of ${\mathbb Z}_N^m$ has Fourier ratio one, so by Lemma \ref{lem:products} we have $\alpha=d-1$ and $\kappa=\frac{m}{2}$, and therefore $p_c(\alpha,\kappa)=\frac{2n}{n-1}$. Writing the frequency as $(\xi,\eta)$, with $\xi\in{\mathbb Z}_N^m$ and $\eta\in{\mathbb Z}_N^n$, we have
$$
|\widehat{1_{S_N}}(\xi,\eta)|=N^{\frac{d}{2}-1}1_{\{(0,0)\}}(\xi,\eta)+N^{\frac{m-1}{2}}1_{\{\xi=0\text{ and }\eta_n\neq0\}}(\xi,\eta).
$$
At $p_c=\frac{2n}{n-1}$ the origin contributes $N^{p_c(\frac{d}{2}-1)}$ to the $p_c$-th moment, and the remaining frequencies contribute $(N-1)N^{n-1}N^{\frac{p_c(m-1)}{2}}$, which is exactly $1-\frac{1}{N}$ times as much. Hence $\|\widehat{1_{S_N}}\|_{p_c}\asymp\|\widehat{1_{S_N}}\|_\infty$, and dividing $1_{S_N}$ by its Fourier $\ell^{p_c}$ norm produces a sequence with norm one whose largest coefficient is bounded below. Synthesis therefore fails at the endpoint for these geometric parameters, and the construction of Section \ref{sec:random} will remove the restrictions on the parameters.

\section{Arithmetic synthesis on paraboloids}\label{sec:arithmetic}

We now let $N$ be any odd modulus. The affine symmetries are still available, but the quadratic coefficient of a Gauss sum need no longer be invertible, and the size of the sum depends on the divisor that this coefficient shares with $N$. Keeping track of this divisor leads to the exact synthesis constant.

\begin{lemma}\label{lem:gauss}
Let $N$ be odd and let $a,b\in{\mathbb Z}_N$. Put $g=\gcd(a,N)$, with $\gcd(0,N)=N$, and $M=\frac{N}{g}$. Then
$$
\frac{1}{N}\left|\sum_{t\bmod N}e^{-\frac{2\pi i}{N}(at^2+bt)}\right|
$$
equals $M^{-\frac{1}{2}}$ if $g\mid b$ and vanishes if $g\nmid b$.
\end{lemma}

\begin{proof}
When $a=0$ the assertion is the orthogonality of characters, so let $a\neq0$ and write $a=ga_0$ and $N=gM$, so that $\gcd(a_0,M)=1$. Every residue $t$ can be written uniquely as $t=u+jM$ with $u\bmod M$ and $0\leq j<g$, and $at^2\equiv ga_0u^2\pmod N$. The sum over $j$ is therefore a sum of $e^{-\frac{2\pi ibj}{g}}$ over $0\leq j<g$. It vanishes unless $g\mid b$, and if $b=gb_0$ the original sum equals
$$
g\sum_{u\bmod M}e^{-\frac{2\pi i}{M}(a_0u^2+b_0u)}.
$$
Since $M$ is odd and $a_0$ is invertible modulo $M$, completing the square removes the linear term without changing the magnitude, and it remains to compute $|G_M(a_0)|$, where $G_M(a_0)=\sum_ue^{-\frac{2\pi ia_0u^2}{M}}$. Now
$$
|G_M(a_0)|^2=\sum_{u,v\bmod M}e^{-\frac{2\pi ia_0(u^2-v^2)}{M}}=\sum_{h,k\bmod M}e^{-\frac{2\pi ia_0hk}{M}}=M,
$$
because the change of variables $h=u-v$, $k=u+v$ is a bijection when $2$ is invertible, and the sum over $k$ vanishes unless $h=0$, in which case it equals $M$. The original sum therefore has magnitude $gM^{\frac{1}{2}}$, and dividing by $N=gM$ we obtain the formula. This is the standard evaluation of quadratic Gauss sums, as in \cite{IrelandRosen}.
\end{proof}

\begin{theorem}\label{thm:paraboloid-exact}
Let $d\geq2$, let $N$ be odd, and let $P_N^{(d)}=\{(t,|t|^2):t\in{\mathbb Z}_N^{d-1}\}$. For $1\leq r<\infty$ define
\begin{equation}\label{eq:divisor-sum}
D_{N,d}(r)=\sum_{M\mid N}\varphi(M)M^{(d-1)(1-\frac{r}{2})}.
\end{equation}
Then
\begin{equation}\label{eq:exact-divisor-constant}
\mathcal C_r(P_N^{(d)})=D_{N,d}(r)^{-\frac{1}{r}}.
\end{equation}
In particular, uniform synthesis at exponent $r$ holds along a family of odd moduli if and only if $D_{N,d}(r)\to\infty$ along that family.
\end{theorem}

\begin{proof}
We normalize the Fourier transform of the indicator by its value at the origin,
$$
F_N(\xi)=\frac{\widehat{1_{P_N^{(d)}}}(\xi)}{N^{\frac{d}{2}-1}}=\frac{1}{N^{d-1}}\sum_{t\in{\mathbb Z}_N^{d-1}}e^{-\frac{2\pi i}{N}(\xi'\cdot t+\xi_d|t|^2)},
$$
where $\xi=(\xi',\xi_d)$. Fix $\xi_d$ and put $g=\gcd(\xi_d,N)$ and $M=\frac{N}{g}$. The sum factors over the $d-1$ coordinates of $t$, and Lemma \ref{lem:gauss}, applied to each factor, tells us that $|F_N(\xi)|=M^{-\frac{d-1}{2}}$ if $g$ divides $\xi_j$ for $1\leq j\leq d-1$, while $F_N(\xi)=0$ otherwise. For each divisor $M$ of $N$ there are $\varphi(M)$ values of $\xi_d$ with $\frac{N}{\gcd(\xi_d,N)}=M$, and $M^{d-1}$ values of $\xi'$ satisfying the divisibility condition. The origin corresponds to $M=1$, where $\varphi(1)=1$. Summing over the divisors, we get
$$
\sum_\xi|F_N(\xi)|^r=\sum_{M\mid N}\varphi(M)M^{d-1}M^{-\frac{r(d-1)}{2}}=D_{N,d}(r).
$$
Since the maps \eqref{eq:paraboloid-symmetry} act transitively on the paraboloid, the formula \eqref{eq:exact-symmetry} yields $\mathcal C_r(P_N^{(d)})=\|F_N\|_r^{-1}$, which is \eqref{eq:exact-divisor-constant}, and the criterion for synthesis follows from the equivalence stated after \eqref{eq:synthesis-constant}.
\end{proof}

The divisor sum \eqref{eq:divisor-sum} is the Fourier moment of the constant function that Hickman and Wright consider in Section 2 of \cite{HickmanWright}. Theorem \ref{thm:paraboloid-exact} identifies its reciprocal power as the optimal synthesis constant for every function supported on the paraboloid, and at even exponents the same sum is the exact extension norm of Theorem \ref{thm:intro-extension}, as we shall see in Section \ref{sec:extension}.

\subsection{Why the exponents are different}

For the parabola $P_N^{(2)}$, each divisor $M$ of $N$ accounts for exactly $M\varphi(M)$ coefficients of $F_N$, all of magnitude $M^{-\frac{1}{2}}$, and so it contributes $\varphi(M)M^{1-\frac{r}{2}}$ to the $r$-th moment. Figure \ref{fig:divisor-structure} shows this structure for a prime power and for a product of two primes. The way in which these contributions accumulate explains the difference between the exponents $4$ and $6$.

\begin{figure}[htbp]
\centering
\begin{tikzpicture}
\begin{scope}[xshift=0.0cm]
\fill[black!100] (0.00,0.00) rectangle (0.42,0.42);
\fill[black!33] (0.00,0.42) rectangle (0.42,0.84);
\fill[black!33] (0.42,0.42) rectangle (0.84,0.84);
\fill[black!33] (0.84,0.42) rectangle (1.26,0.84);
\fill[black!33] (1.26,0.42) rectangle (1.68,0.84);
\fill[black!33] (1.68,0.42) rectangle (2.10,0.84);
\fill[black!33] (2.10,0.42) rectangle (2.52,0.84);
\fill[black!33] (2.52,0.42) rectangle (2.94,0.84);
\fill[black!33] (2.94,0.42) rectangle (3.36,0.84);
\fill[black!33] (3.36,0.42) rectangle (3.78,0.84);
\fill[black!33] (0.00,0.84) rectangle (0.42,1.26);
\fill[black!33] (0.42,0.84) rectangle (0.84,1.26);
\fill[black!33] (0.84,0.84) rectangle (1.26,1.26);
\fill[black!33] (1.26,0.84) rectangle (1.68,1.26);
\fill[black!33] (1.68,0.84) rectangle (2.10,1.26);
\fill[black!33] (2.10,0.84) rectangle (2.52,1.26);
\fill[black!33] (2.52,0.84) rectangle (2.94,1.26);
\fill[black!33] (2.94,0.84) rectangle (3.36,1.26);
\fill[black!33] (3.36,0.84) rectangle (3.78,1.26);
\fill[black!58] (0.00,1.26) rectangle (0.42,1.68);
\fill[black!58] (1.26,1.26) rectangle (1.68,1.68);
\fill[black!58] (2.52,1.26) rectangle (2.94,1.68);
\fill[black!33] (0.00,1.68) rectangle (0.42,2.10);
\fill[black!33] (0.42,1.68) rectangle (0.84,2.10);
\fill[black!33] (0.84,1.68) rectangle (1.26,2.10);
\fill[black!33] (1.26,1.68) rectangle (1.68,2.10);
\fill[black!33] (1.68,1.68) rectangle (2.10,2.10);
\fill[black!33] (2.10,1.68) rectangle (2.52,2.10);
\fill[black!33] (2.52,1.68) rectangle (2.94,2.10);
\fill[black!33] (2.94,1.68) rectangle (3.36,2.10);
\fill[black!33] (3.36,1.68) rectangle (3.78,2.10);
\fill[black!33] (0.00,2.10) rectangle (0.42,2.52);
\fill[black!33] (0.42,2.10) rectangle (0.84,2.52);
\fill[black!33] (0.84,2.10) rectangle (1.26,2.52);
\fill[black!33] (1.26,2.10) rectangle (1.68,2.52);
\fill[black!33] (1.68,2.10) rectangle (2.10,2.52);
\fill[black!33] (2.10,2.10) rectangle (2.52,2.52);
\fill[black!33] (2.52,2.10) rectangle (2.94,2.52);
\fill[black!33] (2.94,2.10) rectangle (3.36,2.52);
\fill[black!33] (3.36,2.10) rectangle (3.78,2.52);
\fill[black!58] (0.00,2.52) rectangle (0.42,2.94);
\fill[black!58] (1.26,2.52) rectangle (1.68,2.94);
\fill[black!58] (2.52,2.52) rectangle (2.94,2.94);
\fill[black!33] (0.00,2.94) rectangle (0.42,3.36);
\fill[black!33] (0.42,2.94) rectangle (0.84,3.36);
\fill[black!33] (0.84,2.94) rectangle (1.26,3.36);
\fill[black!33] (1.26,2.94) rectangle (1.68,3.36);
\fill[black!33] (1.68,2.94) rectangle (2.10,3.36);
\fill[black!33] (2.10,2.94) rectangle (2.52,3.36);
\fill[black!33] (2.52,2.94) rectangle (2.94,3.36);
\fill[black!33] (2.94,2.94) rectangle (3.36,3.36);
\fill[black!33] (3.36,2.94) rectangle (3.78,3.36);
\fill[black!33] (0.00,3.36) rectangle (0.42,3.78);
\fill[black!33] (0.42,3.36) rectangle (0.84,3.78);
\fill[black!33] (0.84,3.36) rectangle (1.26,3.78);
\fill[black!33] (1.26,3.36) rectangle (1.68,3.78);
\fill[black!33] (1.68,3.36) rectangle (2.10,3.78);
\fill[black!33] (2.10,3.36) rectangle (2.52,3.78);
\fill[black!33] (2.52,3.36) rectangle (2.94,3.78);
\fill[black!33] (2.94,3.36) rectangle (3.36,3.78);
\fill[black!33] (3.36,3.36) rectangle (3.78,3.78);
\draw[gray!60,line width=0.2pt,step=0.42] (0,0) grid (3.78,3.78);
\draw[line width=0.5pt] (0,0) rectangle (3.78,3.78);
\node[above,font=\small] at (1.89,3.83) {$N=9$};
\node[below,font=\small] at (1.89,-0.05) {$\xi_1$};
\node[left,font=\small] at (-0.05,1.89) {$\xi_2$};
\end{scope}
\begin{scope}[xshift=5.6cm]
\fill[black!100] (0.00,0.00) rectangle (0.28,0.28);
\fill[black!26] (0.00,0.28) rectangle (0.28,0.56);
\fill[black!26] (0.28,0.28) rectangle (0.56,0.56);
\fill[black!26] (0.56,0.28) rectangle (0.84,0.56);
\fill[black!26] (0.84,0.28) rectangle (1.12,0.56);
\fill[black!26] (1.12,0.28) rectangle (1.40,0.56);
\fill[black!26] (1.40,0.28) rectangle (1.68,0.56);
\fill[black!26] (1.68,0.28) rectangle (1.96,0.56);
\fill[black!26] (1.96,0.28) rectangle (2.24,0.56);
\fill[black!26] (2.24,0.28) rectangle (2.52,0.56);
\fill[black!26] (2.52,0.28) rectangle (2.80,0.56);
\fill[black!26] (2.80,0.28) rectangle (3.08,0.56);
\fill[black!26] (3.08,0.28) rectangle (3.36,0.56);
\fill[black!26] (3.36,0.28) rectangle (3.64,0.56);
\fill[black!26] (3.64,0.28) rectangle (3.92,0.56);
\fill[black!26] (3.92,0.28) rectangle (4.20,0.56);
\fill[black!26] (0.00,0.56) rectangle (0.28,0.84);
\fill[black!26] (0.28,0.56) rectangle (0.56,0.84);
\fill[black!26] (0.56,0.56) rectangle (0.84,0.84);
\fill[black!26] (0.84,0.56) rectangle (1.12,0.84);
\fill[black!26] (1.12,0.56) rectangle (1.40,0.84);
\fill[black!26] (1.40,0.56) rectangle (1.68,0.84);
\fill[black!26] (1.68,0.56) rectangle (1.96,0.84);
\fill[black!26] (1.96,0.56) rectangle (2.24,0.84);
\fill[black!26] (2.24,0.56) rectangle (2.52,0.84);
\fill[black!26] (2.52,0.56) rectangle (2.80,0.84);
\fill[black!26] (2.80,0.56) rectangle (3.08,0.84);
\fill[black!26] (3.08,0.56) rectangle (3.36,0.84);
\fill[black!26] (3.36,0.56) rectangle (3.64,0.84);
\fill[black!26] (3.64,0.56) rectangle (3.92,0.84);
\fill[black!26] (3.92,0.56) rectangle (4.20,0.84);
\fill[black!45] (0.00,0.84) rectangle (0.28,1.12);
\fill[black!45] (0.84,0.84) rectangle (1.12,1.12);
\fill[black!45] (1.68,0.84) rectangle (1.96,1.12);
\fill[black!45] (2.52,0.84) rectangle (2.80,1.12);
\fill[black!45] (3.36,0.84) rectangle (3.64,1.12);
\fill[black!26] (0.00,1.12) rectangle (0.28,1.40);
\fill[black!26] (0.28,1.12) rectangle (0.56,1.40);
\fill[black!26] (0.56,1.12) rectangle (0.84,1.40);
\fill[black!26] (0.84,1.12) rectangle (1.12,1.40);
\fill[black!26] (1.12,1.12) rectangle (1.40,1.40);
\fill[black!26] (1.40,1.12) rectangle (1.68,1.40);
\fill[black!26] (1.68,1.12) rectangle (1.96,1.40);
\fill[black!26] (1.96,1.12) rectangle (2.24,1.40);
\fill[black!26] (2.24,1.12) rectangle (2.52,1.40);
\fill[black!26] (2.52,1.12) rectangle (2.80,1.40);
\fill[black!26] (2.80,1.12) rectangle (3.08,1.40);
\fill[black!26] (3.08,1.12) rectangle (3.36,1.40);
\fill[black!26] (3.36,1.12) rectangle (3.64,1.40);
\fill[black!26] (3.64,1.12) rectangle (3.92,1.40);
\fill[black!26] (3.92,1.12) rectangle (4.20,1.40);
\fill[black!58] (0.00,1.40) rectangle (0.28,1.68);
\fill[black!58] (1.40,1.40) rectangle (1.68,1.68);
\fill[black!58] (2.80,1.40) rectangle (3.08,1.68);
\fill[black!45] (0.00,1.68) rectangle (0.28,1.96);
\fill[black!45] (0.84,1.68) rectangle (1.12,1.96);
\fill[black!45] (1.68,1.68) rectangle (1.96,1.96);
\fill[black!45] (2.52,1.68) rectangle (2.80,1.96);
\fill[black!45] (3.36,1.68) rectangle (3.64,1.96);
\fill[black!26] (0.00,1.96) rectangle (0.28,2.24);
\fill[black!26] (0.28,1.96) rectangle (0.56,2.24);
\fill[black!26] (0.56,1.96) rectangle (0.84,2.24);
\fill[black!26] (0.84,1.96) rectangle (1.12,2.24);
\fill[black!26] (1.12,1.96) rectangle (1.40,2.24);
\fill[black!26] (1.40,1.96) rectangle (1.68,2.24);
\fill[black!26] (1.68,1.96) rectangle (1.96,2.24);
\fill[black!26] (1.96,1.96) rectangle (2.24,2.24);
\fill[black!26] (2.24,1.96) rectangle (2.52,2.24);
\fill[black!26] (2.52,1.96) rectangle (2.80,2.24);
\fill[black!26] (2.80,1.96) rectangle (3.08,2.24);
\fill[black!26] (3.08,1.96) rectangle (3.36,2.24);
\fill[black!26] (3.36,1.96) rectangle (3.64,2.24);
\fill[black!26] (3.64,1.96) rectangle (3.92,2.24);
\fill[black!26] (3.92,1.96) rectangle (4.20,2.24);
\fill[black!26] (0.00,2.24) rectangle (0.28,2.52);
\fill[black!26] (0.28,2.24) rectangle (0.56,2.52);
\fill[black!26] (0.56,2.24) rectangle (0.84,2.52);
\fill[black!26] (0.84,2.24) rectangle (1.12,2.52);
\fill[black!26] (1.12,2.24) rectangle (1.40,2.52);
\fill[black!26] (1.40,2.24) rectangle (1.68,2.52);
\fill[black!26] (1.68,2.24) rectangle (1.96,2.52);
\fill[black!26] (1.96,2.24) rectangle (2.24,2.52);
\fill[black!26] (2.24,2.24) rectangle (2.52,2.52);
\fill[black!26] (2.52,2.24) rectangle (2.80,2.52);
\fill[black!26] (2.80,2.24) rectangle (3.08,2.52);
\fill[black!26] (3.08,2.24) rectangle (3.36,2.52);
\fill[black!26] (3.36,2.24) rectangle (3.64,2.52);
\fill[black!26] (3.64,2.24) rectangle (3.92,2.52);
\fill[black!26] (3.92,2.24) rectangle (4.20,2.52);
\fill[black!45] (0.00,2.52) rectangle (0.28,2.80);
\fill[black!45] (0.84,2.52) rectangle (1.12,2.80);
\fill[black!45] (1.68,2.52) rectangle (1.96,2.80);
\fill[black!45] (2.52,2.52) rectangle (2.80,2.80);
\fill[black!45] (3.36,2.52) rectangle (3.64,2.80);
\fill[black!58] (0.00,2.80) rectangle (0.28,3.08);
\fill[black!58] (1.40,2.80) rectangle (1.68,3.08);
\fill[black!58] (2.80,2.80) rectangle (3.08,3.08);
\fill[black!26] (0.00,3.08) rectangle (0.28,3.36);
\fill[black!26] (0.28,3.08) rectangle (0.56,3.36);
\fill[black!26] (0.56,3.08) rectangle (0.84,3.36);
\fill[black!26] (0.84,3.08) rectangle (1.12,3.36);
\fill[black!26] (1.12,3.08) rectangle (1.40,3.36);
\fill[black!26] (1.40,3.08) rectangle (1.68,3.36);
\fill[black!26] (1.68,3.08) rectangle (1.96,3.36);
\fill[black!26] (1.96,3.08) rectangle (2.24,3.36);
\fill[black!26] (2.24,3.08) rectangle (2.52,3.36);
\fill[black!26] (2.52,3.08) rectangle (2.80,3.36);
\fill[black!26] (2.80,3.08) rectangle (3.08,3.36);
\fill[black!26] (3.08,3.08) rectangle (3.36,3.36);
\fill[black!26] (3.36,3.08) rectangle (3.64,3.36);
\fill[black!26] (3.64,3.08) rectangle (3.92,3.36);
\fill[black!26] (3.92,3.08) rectangle (4.20,3.36);
\fill[black!45] (0.00,3.36) rectangle (0.28,3.64);
\fill[black!45] (0.84,3.36) rectangle (1.12,3.64);
\fill[black!45] (1.68,3.36) rectangle (1.96,3.64);
\fill[black!45] (2.52,3.36) rectangle (2.80,3.64);
\fill[black!45] (3.36,3.36) rectangle (3.64,3.64);
\fill[black!26] (0.00,3.64) rectangle (0.28,3.92);
\fill[black!26] (0.28,3.64) rectangle (0.56,3.92);
\fill[black!26] (0.56,3.64) rectangle (0.84,3.92);
\fill[black!26] (0.84,3.64) rectangle (1.12,3.92);
\fill[black!26] (1.12,3.64) rectangle (1.40,3.92);
\fill[black!26] (1.40,3.64) rectangle (1.68,3.92);
\fill[black!26] (1.68,3.64) rectangle (1.96,3.92);
\fill[black!26] (1.96,3.64) rectangle (2.24,3.92);
\fill[black!26] (2.24,3.64) rectangle (2.52,3.92);
\fill[black!26] (2.52,3.64) rectangle (2.80,3.92);
\fill[black!26] (2.80,3.64) rectangle (3.08,3.92);
\fill[black!26] (3.08,3.64) rectangle (3.36,3.92);
\fill[black!26] (3.36,3.64) rectangle (3.64,3.92);
\fill[black!26] (3.64,3.64) rectangle (3.92,3.92);
\fill[black!26] (3.92,3.64) rectangle (4.20,3.92);
\fill[black!26] (0.00,3.92) rectangle (0.28,4.20);
\fill[black!26] (0.28,3.92) rectangle (0.56,4.20);
\fill[black!26] (0.56,3.92) rectangle (0.84,4.20);
\fill[black!26] (0.84,3.92) rectangle (1.12,4.20);
\fill[black!26] (1.12,3.92) rectangle (1.40,4.20);
\fill[black!26] (1.40,3.92) rectangle (1.68,4.20);
\fill[black!26] (1.68,3.92) rectangle (1.96,4.20);
\fill[black!26] (1.96,3.92) rectangle (2.24,4.20);
\fill[black!26] (2.24,3.92) rectangle (2.52,4.20);
\fill[black!26] (2.52,3.92) rectangle (2.80,4.20);
\fill[black!26] (2.80,3.92) rectangle (3.08,4.20);
\fill[black!26] (3.08,3.92) rectangle (3.36,4.20);
\fill[black!26] (3.36,3.92) rectangle (3.64,4.20);
\fill[black!26] (3.64,3.92) rectangle (3.92,4.20);
\fill[black!26] (3.92,3.92) rectangle (4.20,4.20);
\draw[gray!60,line width=0.2pt,step=0.28] (0,0) grid (4.20,4.20);
\draw[line width=0.5pt] (0,0) rectangle (4.20,4.20);
\node[above,font=\small] at (2.10,4.25) {$N=15$};
\node[below,font=\small] at (2.10,-0.05) {$\xi_1$};
\node[left,font=\small] at (-0.05,2.10) {$\xi_2$};
\end{scope}
\begin{scope}[xshift=10.6cm]
\fill[black!100] (0,3.40) rectangle (0.32,3.72);
\draw[line width=0.3pt] (0,3.40) rectangle (0.32,3.72);
\node[right,font=\small] at (0.38,3.56) {$M=1$};
\fill[black!58] (0,2.78) rectangle (0.32,3.10);
\draw[line width=0.3pt] (0,2.78) rectangle (0.32,3.10);
\node[right,font=\small] at (0.38,2.94) {$M=3$};
\fill[black!45] (0,2.16) rectangle (0.32,2.48);
\draw[line width=0.3pt] (0,2.16) rectangle (0.32,2.48);
\node[right,font=\small] at (0.38,2.32) {$M=5$};
\fill[black!33] (0,1.54) rectangle (0.32,1.86);
\draw[line width=0.3pt] (0,1.54) rectangle (0.32,1.86);
\node[right,font=\small] at (0.38,1.70) {$M=9$};
\fill[black!26] (0,0.92) rectangle (0.32,1.24);
\draw[line width=0.3pt] (0,0.92) rectangle (0.32,1.24);
\node[right,font=\small] at (0.38,1.08) {$M=15$};
\draw[line width=0.3pt,fill=white] (0,0.30) rectangle (0.32,0.62);
\node[right,font=\small] at (0.38,0.46) {zero};
\end{scope}
\end{tikzpicture}
\caption{The normalized Fourier transform $F_N$ of the parabola $P_N^{(2)}$ for $N=9$ and $N=15$, with $\xi_1$ horizontal and $\xi_2$ vertical, both running from $0$ to $N-1$. A coefficient is nonzero exactly when $g=\gcd(\xi_2,N)$ divides $\xi_1$, and its magnitude is then $M^{-\frac{1}{2}}$, where $M=\frac{N}{g}$, which the shading records. For $N=9$ the divisors $1$, $3$ and $9$ account for $1$, $6$ and $54$ coefficients, and for $N=15$ the divisors $1$, $3$, $5$ and $15$ account for $1$, $6$, $20$ and $120$ coefficients, in agreement with the count $M\varphi(M)$.}
\label{fig:divisor-structure}
\end{figure}

Fix an odd prime $\ell$ and let $N=\ell^k$. The divisors of $N$ are $1,\ell,\ldots,\ell^k$, and the sum \eqref{eq:divisor-sum} becomes
$$
D_{\ell^k,2}(r)=1+\left(1-\frac{1}{\ell}\right)\sum_{j=1}^k\ell^{j(2-\frac{r}{2})}.
$$
At $r=4$ every power of $\ell$ contributes the same amount, so the sum grows linearly in $k$. For $r>4$ the contributions decrease geometrically, and the sum stays bounded as $k$ grows. Along the powers of a fixed prime, synthesis therefore holds up to and including the exponent $4$.

Now let $N$ be a product of distinct odd primes. The Chinese remainder theorem identifies the parabola modulo $N$ with the Cartesian product of the parabolas modulo the prime divisors of $N$. Under the corresponding identification of characters the normalized Fourier transform factors, and its moments multiply. In terms of the divisor sum,
$$
D_{N,2}(r)=\prod_{\substack{\ell\mid N\\ \ell\text{ prime}}}\left(1+\left(1-\frac{1}{\ell}\right)\ell^{2-\frac{r}{2}}\right).
$$
For fixed $r>4$ every factor exceeds one by at most $3^{2-\frac{r}{2}}<1$, and taking logarithms we find that
$$
\log D_{N,2}(r)\asymp_r\sum_{\substack{\ell\mid N\\ \ell\text{ prime}}}\ell^{2-\frac{r}{2}}.
$$
At $r=6$ the right side is the sum of the reciprocals of the prime divisors of $N$. It diverges when $N$ contains all odd primes up to a growing bound, while for $r>6$ it stays bounded even if all primes are included. The change in the exponent thus comes from replacing a geometric sum over the powers of one prime by a sum over different primes. The distinction between these divisor sums already appears in \cite[Example 2.1]{HickmanWright}, and Corollaries \ref{cor:prime-power-endpoint} and \ref{cor:primorial} below contain the precise statements in every dimension.

At the exponent $6$ we can say exactly which families of odd moduli have synthesis, repeated prime factors included.

\begin{corollary}\label{cor:parabola-six}
Let $N>1$ be odd, and write $N=\prod_{\ell\mid N}\ell^{a_\ell}$, where the product is over the distinct prime divisors of $N$ and $a_\ell\geq1$. Then
\begin{equation}\label{eq:parabola-six-product}
D_{N,2}(6)=\prod_{\ell\mid N}\left(1+\frac{1-\ell^{-a_\ell}}{\ell}\right)
\end{equation}
and
\begin{equation}\label{eq:parabola-six-log}
\frac{1}{2}\sum_{\ell\mid N}\frac{1}{\ell}\leq\log D_{N,2}(6)\leq\sum_{\ell\mid N}\frac{1}{\ell}.
\end{equation}
Consequently, uniform synthesis on $P_N^{(2)}$ at exponent $6$ holds along a family of odd moduli if and only if
$$
\sum_{\ell\mid N}\frac{1}{\ell}\longrightarrow\infty
$$
along that family.
\end{corollary}

\begin{proof}
The divisor sum is a multiplicative function of $N$. Summing the geometric series in each factor, we find that
$$
D_{N,2}(6)=\prod_{\ell\mid N}\left(1+\left(1-\frac{1}{\ell}\right)\sum_{j=1}^{a_\ell}\ell^{-j}\right)=\prod_{\ell\mid N}\left(1+\frac{1-\ell^{-a_\ell}}{\ell}\right).
$$
Put $u_\ell=\frac{1-\ell^{-a_\ell}}{\ell}$. Since $\ell\geq3$ and $a_\ell\geq1$, we have $\frac{2}{3\ell}\leq u_\ell\leq\frac{1}{\ell}\leq\frac{1}{3}$, and the inequalities $\frac{u}{1+u}\leq\log(1+u)\leq u$ for $u\geq0$ imply that
$$
\frac{1}{2\ell}\leq\log(1+u_\ell)\leq\frac{1}{\ell}.
$$
Summing over the prime divisors proves \eqref{eq:parabola-six-log}, and the criterion for synthesis follows because $\mathcal C_6(P_N^{(2)})=D_{N,2}(6)^{-\frac{1}{6}}$.
\end{proof}

For a single prime the formula reads
$$
D_{\ell^k,2}(6)=1+\frac{1-\ell^{-k}}{\ell}.
$$
All the powers of $\ell$ together add at most $\frac{1}{\ell}$ to the contribution of the origin, while each new prime factor $\ell$ increases the logarithm of the moment by an amount comparable to $\frac{1}{\ell}$. Such small contributions can have a divergent sum, and synthesis only requires the moment to tend to infinity.

Distinct prime factors alone do not guarantee synthesis at $6$. If we choose increasing odd primes $\ell_j\geq2^j$ and put $N_k=\prod_{j=1}^k\ell_j$, then
$$
\sum_{\ell\mid N_k}\frac{1}{\ell}=\sum_{j=1}^k\frac{1}{\ell_j}\leq\sum_{j=1}^{\infty}2^{-j}=1,
$$
and by Corollary \ref{cor:parabola-six} synthesis at $6$ fails along this family.

\subsection{The same concentration exponent for all odd moduli}

\begin{proposition}\label{prop:odd-kappa}
As $N$ tends to infinity through any family of odd moduli,
$$
\operatorname{FR}(1_{P_N^{(d)}})=N^{\frac{d}{2}-o(1)}\text{ and }\kappa_N(1_{P_N^{(d)}})\longrightarrow0.
$$
\end{proposition}

\begin{proof}
By the moment calculation at $r=1$ and Plancherel,
$$
\operatorname{FR}(1_{P_N^{(d)}})=\frac{N^{\frac{d}{2}-1}D_{N,d}(1)}{N^{\frac{d-1}{2}}}=N^{-\frac{1}{2}}\sum_{M\mid N}\varphi(M)M^{\frac{d-1}{2}}.
$$
The term $M=N$ alone contributes $\varphi(N)N^{\frac{d-2}{2}}$, and the Cauchy--Schwarz inequality provides the upper bound $N^{\frac{d}{2}}$, so it remains to recall that $\varphi(N)=N^{1-o(1)}$.

Here is an elementary argument for this fact. Let $\omega(N)$ be the number of distinct prime divisors of $N$. If $\omega(N)=k$, the product of these primes is at least $(k+1)!$, since the $j$-th smallest of them is at least $j+1$, and the estimate $\log((k+1)!)\gtrsim k\log(k+1)$ implies that $\omega(N)=O(\frac{\log N}{\log\log N})$. Since
$$
\frac{\varphi(N)}{N}=\prod_{\ell\mid N}\left(1-\frac{1}{\ell}\right)\geq2^{-\omega(N)},
$$
we conclude that $\varphi(N)\geq N^{1-o(1)}$, and the two bounds for the Fourier ratio prove the proposition.
\end{proof}

The concentration exponent does not control the decay of individual coefficients away from the origin. In dimension two, fix an odd prime $\ell$ and put $N=\ell^k$. At the nonzero frequency $(0,\ell^{k-1})$ we have
$$
\widehat{1_{P_N^{(2)}}}(0,\ell^{k-1})=\frac{1}{\ell}\sum_{t\bmod\ell}e^{-\frac{2\pi it^2}{\ell}},
$$
a number of magnitude $\ell^{-\frac{1}{2}}$ that does not tend to zero, whereas along primes every coefficient with nonzero second coordinate has magnitude $N^{-\frac{1}{2}}$. The concentration exponent is zero in both cases.

\subsection{Prime powers and the critical exponent}

We set
$$
r_c=\frac{2d}{d-1},
$$
the universal critical exponent $p_c(\alpha,\kappa)$ for the parameters $\alpha=d-1$ and $\kappa=0$ of the paraboloid.

\begin{corollary}\label{cor:prime-power-endpoint}
Fix an odd prime $\ell$ and let $N=\ell^k$. Uniform synthesis on $P_N^{(d)}$ holds exactly for $1\leq r\leq r_c$. More precisely, with implied constants depending on $\ell$, $d$ and $r$,
$$
\mathcal C_r(P_{\ell^k}^{(d)})\asymp N^{\frac{d-1}{2}-\frac{d}{r}}\text{ for }1\leq r<r_c\text{ and }\mathcal C_r(P_{\ell^k}^{(d)})\asymp1\text{ for }r>r_c,
$$
while at the endpoint
\begin{equation}\label{eq:prime-power-critical}
\mathcal C_{r_c}(P_{\ell^k}^{(d)})=\left(1+k\left(1-\frac{1}{\ell}\right)\right)^{-\frac{1}{r_c}},
\end{equation}
which is comparable to $(\log N)^{-\frac{1}{r_c}}$.
\end{corollary}

\begin{proof}
Write $s=(d-1)(\frac{r}{2}-1)$. Since $\varphi(\ell^j)=(1-\frac{1}{\ell})\ell^j$ for $j\geq1$,
$$
D_{\ell^k,d}(r)=1+\left(1-\frac{1}{\ell}\right)\sum_{j=1}^k\ell^{j(1-s)}.
$$
For $r<r_c$ we have $s<1$, the geometric sum is comparable to $N^{1-s}$, and its reciprocal $r$-th root is comparable to $N^{-\frac{1-s}{r}}=N^{\frac{d-1}{2}-\frac{d}{r}}$. At $r=r_c$ we have $s=1$ and every summand equals one, which is the formula \eqref{eq:prime-power-critical}. For $r>r_c$ the geometric sum is bounded independently of $k$, while the divisor sum is at least one because of its first term. All the assertions now follow from Theorem \ref{thm:paraboloid-exact}.
\end{proof}

For a general odd modulus we write $\ell^a\parallel N$ when $\ell^a$ divides $N$ and $\ell^{a+1}$ does not. Since $(d-1)(1-\frac{r_c}{2})=-1$, factoring the divisor sum over the prime divisors of $N$ leads to
\begin{equation}\label{eq:endpoint-product}
D_{N,d}(r_c)=\sum_{M\mid N}\frac{\varphi(M)}{M}=\prod_{\ell^a\parallel N}\left(1+a\left(1-\frac{1}{\ell}\right)\right).
\end{equation}
Let $\tau(N)=\prod_{\ell^a\parallel N}(a+1)$ be the number of positive divisors of $N$. Then
\begin{equation}\label{eq:tau-comparison}
\tau(N)^{\frac{2}{3}}\leq D_{N,d}(r_c)\leq\tau(N),
\end{equation}
because $\ell\geq3$ implies $1+a(1-\frac{1}{\ell})\geq1+\frac{2}{3}a\geq(1+a)^{\frac{2}{3}}$, the last inequality coming from a comparison of the concave function $x^{\frac{2}{3}}$ with its tangent line at $x=1$, while the upper bound is immediate. Synthesis at the endpoint $r_c$ therefore holds along a family of odd moduli exactly when $\tau(N)\to\infty$ along that family.

\subsection{Products of distinct primes}

The number of prime factors can push synthesis beyond $r_c$. Let
$$
N_Y=\prod_{\substack{3\leq\ell\leq Y\\ \ell\text{ prime}}}\ell\text{ and }r_*=\frac{2(d+1)}{d-1}.
$$
The family $\{N_Y\}$ reaches the larger exponent $r_*$, and no family of odd moduli does better.

\begin{corollary}\label{cor:primorial}
Along $N=N_Y$, uniform synthesis on $P_N^{(d)}$ holds exactly for $1\leq r\leq r_*$. For every fixed $r>r_*$ there is a constant $c_{d,r}>0$ such that
$$
\mathcal C_r(P_N^{(d)})\geq c_{d,r}
$$
for every odd $N$. Thus no family of odd moduli can yield uniform synthesis on these paraboloids at an exponent greater than $r_*$.
\end{corollary}

\begin{proof}
Again put $s=(d-1)(\frac{r}{2}-1)$. Since $N_Y$ is a product of distinct primes,
$$
D_{N_Y,d}(r)=\prod_{\substack{3\leq\ell\leq Y\\ \ell\text{ prime}}}\left(1+\left(1-\frac{1}{\ell}\right)\ell^{1-s}\right).
$$
If $r\leq r_*$, then $s\leq2$, each factor is at least $1+\frac{2}{3\ell}$, and the product tends to infinity because the sum of the reciprocals of the primes diverges.

For completeness we recall Euler's argument for this divergence. If $\sum_\ell\frac{1}{\ell}$ converged, the products $\prod_{\ell\leq Y}(1-\frac{1}{\ell})^{-1}$ would be bounded, because $-\log(1-\frac{1}{\ell})\leq\frac{1}{\ell}+C\ell^{-2}$. On the other hand, expanding each factor as a geometric series, we see that the product is at least $\sum_{n\leq Y}\frac{1}{n}$, since every integer $n\leq Y$ has all its prime divisors at most $Y$, and this contradicts the divergence of the harmonic series. Removing the prime $2$ does not affect the divergence, and since $\log(1+\frac{2}{3\ell})\gtrsim\frac{1}{\ell}$, the product above tends to infinity.

If $r>r_*$, then $s>2$, and for every odd $N$
$$
D_{N,d}(r)=\sum_{M\mid N}\varphi(M)M^{-s}\leq\sum_{M=1}^{\infty}M^{1-s}<\infty.
$$
This bound does not depend on $N$, and Theorem \ref{thm:paraboloid-exact} completes the proof.
\end{proof}

We can now collect the synthesis ranges for the paraboloid $P_N^{(d)}$ in every dimension $d\geq2$.
\begin{center}
\renewcommand{\arraystretch}{1.8}
\begin{tabular}{ll}
\hline
Moduli tending to infinity & Exact synthesis range \\
\hline
Odd primes & $1\leq r<\frac{2d}{d-1}$ \\
Powers of a fixed odd prime & $1\leq r\leq\frac{2d}{d-1}$ \\
Products $N_Y$ of all odd primes at most $Y$ & $1\leq r\leq\frac{2(d+1)}{d-1}$ \\
\hline
\end{tabular}
\end{center}
In dimension two, $r_c=4$ and $r_*=6$, and the divergence at the larger endpoint also follows from Corollary \ref{cor:parabola-six}. In dimensions three and four the two exponents are $3$ and $4$, and $\frac{8}{3}$ and $\frac{10}{3}$, respectively.

\section{Random constructions and sharpness for every pair}\label{sec:random}

The geometric examples only realize the parameters allowed by the dimensions of their factors. To obtain arbitrary parameters, we choose a random collection of cosets of a subgroup. Recall that a coset of a subgroup $H$ of an abelian group $G$ is a translate $x+H$, and that the cosets form the quotient group $G/H$, with the addition inherited from $G$. We begin with the probability estimate needed for the construction, writing ${\mathbb E}$ for expectation and ${\mathbb P}$ for probability.

\begin{lemma}\label{lem:random-moments}
Let $G$ be a finite abelian group of cardinality $L\geq2$, and choose $R$ uniformly among its $M$-element subsets, where $1\leq M<L$. For every nontrivial character $\chi$ of $G$ and every $r>0$,
\begin{equation}\label{eq:random-moment}
{\mathbb E}\left|\sum_{x\in R}\chi(x)\right|^r\leq C_rM^{\frac{r}{2}},
\end{equation}
where $C_r$ is independent of $G$, $M$ and $\chi$.
\end{lemma}

\begin{proof}
A nontrivial character has sum zero, since if $\chi(y)\neq1$, translating the sum by $y$ shows that $\sum_x\chi(x)=\chi(y)\sum_x\chi(x)$. Let $a(x)$ be either the real or the imaginary part of $\chi(x)$, so that $a(x)\in[-1,1]$ and $\sum_{x\in G}a(x)=0$. We choose distinct points $X_1,\ldots,X_M$ one at a time, each uniformly from the population that remains, and in this way produce a uniformly distributed $M$-element subset. Write $S_j=\sum_{i=1}^ja(X_i)$. After the first $j$ choices the remaining population has sum $-S_j$, so the conditional expectation of the final sum is
$$
Z_j={\mathbb E}(S_M\mid X_1,\ldots,X_j)=S_j-\frac{M-j}{L-j}S_j=\frac{L-M}{L-j}S_j
$$
for $0\leq j\leq M$. In particular $Z_0=0$ and $Z_M=S_M$. If $\mu_{j-1}=-\frac{S_{j-1}}{L-j+1}$ is the average of $a$ over the population remaining before the $j$-th choice, a direct calculation shows that
$$
Z_j-Z_{j-1}=\frac{L-M}{L-j}\left(a(X_j)-\mu_{j-1}\right).
$$
This increment has conditional mean zero, and its absolute value is at most $2$, because $\frac{L-M}{L-j}\leq1$.

A random variable $Y$ with mean zero and $|Y|\leq2$ satisfies
$$
{\mathbb E}e^{tY}\leq\cosh(2t)\leq e^{2t^2}.
$$
For the first inequality, bound the exponential on $[-2,2]$ by its secant line and take expectations, and for the second, compare power series and use $(2j)!\geq2^jj!$. Applying this estimate conditionally and iterating over $j$, we get ${\mathbb E}e^{tS_M}\leq e^{2Mt^2}$ for every real $t$, and Markov's inequality with $t=\frac{u}{4M}$ implies that
$$
{\mathbb P}(|S_M|>u)\leq2e^{-\frac{u^2}{8M}}.
$$
If a complex number has magnitude greater than $u$, its real or its imaginary part has magnitude greater than $\frac{u}{\sqrt2}$, so
$$
{\mathbb P}\left(\left|\sum_{x\in R}\chi(x)\right|>u\right)\leq4e^{-\frac{u^2}{16M}}.
$$
Finally, the identity ${\mathbb E}W^r=r\int_0^\infty u^{r-1}{\mathbb P}(W>u)du$ for $W\geq0$ and the change of variables $u=M^{\frac{1}{2}}v$ prove \eqref{eq:random-moment}. This is a standard form of concentration for sampling without replacement, as in \cite{Serfling}.
\end{proof}

\subsection{Endpoint counterexamples in a prescribed dimension}

Let us first see how the parameters should be chosen. A union of cosets of a subgroup of size $h$ has Fourier transform supported on at most $\frac{N^d}{h}$ frequencies. To arrange a Fourier ratio of size $N^{\frac{d}{2}-\kappa}$, we therefore take $h$ of size $N^{2\kappa}$, and the desired size of the support then requires about $N^{\alpha-2\kappa}$ cosets. Randomness will make the Fourier ratio comparable to the square root of the number of available frequencies.

We also need an elementary fact about subgroups. If $N=\ell^k$ and $b$ is an integer with $0\leq b\leq dk$, then ${\mathbb Z}_N^d$ has a subgroup of cardinality $\ell^b$. Indeed, we may write $b=b_1+\cdots+b_d$ with $0\leq b_j\leq k$ and take the product of the coordinate subgroups $\ell^{k-b_j}{\mathbb Z}_N$, each of which has $\ell^{b_j}$ elements.

\begin{theorem}\label{thm:all-pair-sharpness}
Fix $d\geq1$, an odd prime $\ell$ and parameters $0<\alpha<d$ and $0\leq\kappa<\frac{\alpha}{2}$. Along $N=\ell^k$ there are sets $S_N\subseteq{\mathbb Z}_N^d$ with
$$
|S_N|\asymp N^\alpha\text{ and }\operatorname{FR}(1_{S_N})\asymp N^{\frac{d}{2}-\kappa},
$$
and scalar multiples $f_N$ of $1_{S_N}$ such that $\|\widehat f_N\|_\infty=1$ and $\sup_N\|\widehat f_N\|_{p_c}<\infty$, where
$$
p_c=\frac{2(d-2\kappa)}{\alpha-2\kappa}.
$$
Thus the universal critical exponent $p_c$ is sharp for every admissible pair with $\kappa<\frac{\alpha}{2}$ in the fixed dimension $d$.
\end{theorem}

\begin{proof}
Put $\beta=2\kappa$ and $\gamma=\alpha-2\kappa$. Choose a subgroup $H_N$ of order $h_N=\ell^{\lfloor\beta k\rfloor}\asymp N^\beta$, and let $G_N={\mathbb Z}_N^d/H_N$, a group of order $L_N=\frac{N^d}{h_N}\asymp N^{d-\beta}$. Since $\alpha<d$, we have $\frac{\lfloor N^\gamma\rfloor}{L_N}\to0$, so for large $N$ we may choose a set $R_N$ of $M_N=\lfloor N^\gamma\rfloor$ points uniformly at random in $G_N$. Let $S_N$ be the full inverse image of $R_N$ in ${\mathbb Z}_N^d$. It has $h_NM_N\asymp N^\alpha$ elements, and we put
$$
f_N=\frac{N^{\frac{d}{2}}}{h_NM_N}1_{S_N}.
$$

The annihilator of $H_N$ is
$$
H_N^\perp=\{\xi\in{\mathbb Z}_N^d:e^{\frac{2\pi ix\cdot\xi}{N}}=1\text{ for every }x\in H_N\}.
$$
Summing a character over $H_N$, we find that $\widehat{1_{H_N}}(\xi)=N^{-\frac{d}{2}}h_N1_{H_N^\perp}(\xi)$, and Plancherel then forces $h_N=N^{-d}h_N^2|H_N^\perp|$, that is, $|H_N^\perp|=L_N$. For $\xi\in H_N^\perp$ the character $x\mapsto e^{-\frac{2\pi ix\cdot\xi}{N}}$ is constant on each coset of $H_N$, so it defines a character $\chi_\xi$ of $G_N$, which is nontrivial unless $\xi=0$. Summing over each coset, we see that $\widehat f_N$ vanishes outside $H_N^\perp$, while on the annihilator
\begin{equation}\label{eq:quotient-fourier}
\widehat f_N(\xi)=\frac{1}{M_N}\sum_{y\in R_N}\chi_\xi(y).
\end{equation}
In particular $\widehat f_N(0)=1$ and $\|\widehat f_N\|_\infty=1$.

Since $p_c=\frac{2(d-\beta)}{\gamma}>2$, we have $L_N\asymp M_N^{\frac{p_c}{2}}$, and Lemma \ref{lem:random-moments}, applied to the $L_N-1$ nontrivial characters $\chi_\xi$, yields
$$
{\mathbb E}\|\widehat f_N\|_{p_c}^{p_c}\leq1+C_{p_c}L_NM_N^{-\frac{p_c}{2}}\leq C.
$$
For each modulus we choose a realization whose moment does not exceed this bound, and the assertion about the norms follows.

It remains to compute the Fourier ratio. By Plancherel, $\|\widehat f_N\|_2^2=\frac{N^d}{h_NM_N}=\frac{L_N}{M_N}$, and since the Fourier transform is supported on $L_N$ points, the Cauchy--Schwarz inequality yields $\operatorname{FR}(f_N)\leq L_N^{\frac{1}{2}}$. In the other direction, the inequality \eqref{eq:interpolation-fr} with $p=p_c$ implies that
$$
\operatorname{FR}(f_N)\geq\left(\frac{\|\widehat f_N\|_2}{\|\widehat f_N\|_{p_c}}\right)^{\frac{p_c}{p_c-2}}\gtrsim\left(\frac{L_N}{M_N}\right)^{\frac{p_c}{2(p_c-2)}}\asymp L_N^{\frac{1}{2}},
$$
where the last step uses $M_N\asymp L_N^{\frac{2}{p_c}}$. Scalar multiplication does not change the Fourier ratio, so
$$
\operatorname{FR}(1_{S_N})\asymp L_N^{\frac{1}{2}}\asymp N^{\frac{d}{2}-\kappa}.
$$
At the finitely many initial moduli we may take any nonempty set and normalize its indicator so that its Fourier coefficient at zero equals one, enlarging the norm bound if necessary.
\end{proof}

The same counterexamples work for every $p\geq p_c$, since $\|\widehat f_N\|_p\leq\|\widehat f_N\|_{p_c}$ for counting measure. At $\kappa=0$ the subgroup is trivial, and the proof becomes a direct moment argument for the endpoint of the cardinality estimate, which Bhowmik, Deodhar and Iosevich established using Bourgain's theorem in \cite[Proposition 9]{BhowmikDeodharIosevich}.

At $\kappa=\frac{\alpha}{2}$ a subgroup of cardinality $\ell^{\lfloor\alpha k\rfloor}$ realizes the parameters directly, because its Fourier transform is constant on its annihilator and $\operatorname{FR}(1_{H_N})=(\frac{N^d}{|H_N|})^{\frac{1}{2}}$. There is no finite critical exponent to test in this case. Every admissible pair is therefore realized in each dimension $d>\alpha$ along prime powers, maximal concentration included, and $\alpha$ and $\kappa$ alone cannot guarantee a larger synthesis range. The paraboloids of Section \ref{sec:arithmetic} show that a particular family can nevertheless have one.

\subsection{Building blocks for all integer moduli}

If the ambient dimension is allowed to grow, every admissible pair can be realized along all integer moduli, by combining intervals, which have maximal concentration, with random sets, whose concentration exponent is zero.

\begin{lemma}\label{lem:random-zero}
For each $0<\gamma\leq1$ there are sets $R_N\subseteq{\mathbb Z}_N$ with $|R_N|=N^{\gamma+o(1)}$ and $\operatorname{FR}(1_{R_N})\asymp N^{\frac{1}{2}}$. In particular, their concentration exponents at each modulus tend to zero.
\end{lemma}

\begin{proof}
For $\gamma<1$ put $M_N=\lfloor N^\gamma\rfloor$, for $\gamma=1$ put $M_N=\lfloor\frac{N}{2}\rfloor$, and choose a random $M_N$-element subset of ${\mathbb Z}_N$. By Lemma \ref{lem:random-moments} with $r=4$,
$$
{\mathbb E}\sum_{\xi\neq0}|\widehat{1_{R_N}}(\xi)|^4\leq C\frac{M_N^2}{N},
$$
and we fix a realization satisfying the same bound. The nonzero Fourier coefficients have square sum $M_N-\frac{M_N^2}{N}\gtrsim M_N$, and the logarithmic convexity of the $\ell^r$ norms between $\ell^1$ and $\ell^4$, applied to these coefficients, implies that
$$
\sum_{\xi\neq0}|\widehat{1_{R_N}}(\xi)|\geq\frac{\left(\sum_{\xi\neq0}|\widehat{1_{R_N}}(\xi)|^2\right)^{\frac{3}{2}}}{\left(\sum_{\xi\neq0}|\widehat{1_{R_N}}(\xi)|^4\right)^{\frac{1}{2}}}\gtrsim(NM_N)^{\frac{1}{2}}.
$$
The Cauchy--Schwarz inequality provides the reverse bound for the full $\ell^1$ norm, and dividing by $M_N^{\frac{1}{2}}$ we obtain the claim.
\end{proof}

\begin{lemma}\label{lem:intervals}
For each $0<\beta\leq1$ there are sets $B_N\subseteq{\mathbb Z}_N$ with cardinality exponent $\beta$ whose concentration exponents at each modulus tend to $\frac{\beta}{2}$.
\end{lemma}

\begin{proof}
For $\beta<1$ let $M=\lfloor N^\beta\rfloor$ and $B_N=\{0,1,\ldots,M-1\}$. Summing the geometric series, we find that
$$
|\widehat{1_{B_N}}(\xi)|=N^{-\frac{1}{2}}\left|\frac{\sin(\frac{\pi M\xi}{N})}{\sin(\frac{\pi\xi}{N})}\right|
$$
for $\xi\neq0$. For $1\leq\xi\leq\frac{N}{2}$ the unnormalized sum is at most $C\min\{M,\frac{N}{\xi}\}$, and splitting the sum at $\frac{N}{M}$ and using symmetry, we get
$$
\|\widehat{1_{B_N}}\|_1\lesssim N^{-\frac{1}{2}}\left(M+N+N\sum_{\frac{N}{M}<\xi\leq\frac{N}{2}}\frac{1}{\xi}\right)\lesssim N^{\frac{1}{2}}(1+\log M).
$$
Fourier inversion provides the lower bound $\|\widehat{1_{B_N}}\|_1\geq N^{\frac{1}{2}}$, and by Plancherel $\|\widehat{1_{B_N}}\|_2=M^{\frac{1}{2}}$, so
$$
\left(\frac{N}{M}\right)^{\frac{1}{2}}\leq\operatorname{FR}(1_{B_N})\lesssim\left(\frac{N}{M}\right)^{\frac{1}{2}}(1+\log M),
$$
and taking logarithms we obtain the asserted limit. For $\beta=1$ we take $B_N={\mathbb Z}_N$, whose Fourier ratio is one.
\end{proof}

\begin{corollary}\label{cor:all-moduli-realization}
For every $\alpha>0$ and $0\leq\kappa\leq\frac{\alpha}{2}$ there are a dimension $d$ and sets $S_N\subseteq{\mathbb Z}_N^d$, defined for all integer moduli, with cardinality exponent $\alpha$ and concentration exponent $\kappa$.
\end{corollary}

\begin{proof}
Write $2\kappa$ as a sum of finitely many numbers in $(0,1]$, and do the same for $\alpha-2\kappa$, omitting a sum if the corresponding parameter is zero. For each summand of the first sum we take an interval factor from Lemma \ref{lem:intervals}, and for each summand of the second a random factor from Lemma \ref{lem:random-zero}. The Cartesian product of these factors has the required cardinality exponent. All the concentration exponents at each modulus converge, so by Lemma \ref{lem:products} the limits add, and the product has concentration exponent exactly $\kappa$. A singleton coordinate may be added if a dimension strictly larger than $\alpha$ is desired.
\end{proof}

\subsection{Random supports and restricted isometry}

In Theorem \ref{thm:asymptotic} the exponent $\kappa_f$ of a particular sequence of functions may replace that of the indicators, and a set can carry concentrated functions even when its indicator is not concentrated, for instance when it contains a large subgroup. On random sets of the sizes considered below this cannot happen, since every choice of weights has concentration exponent zero. Our first proof of this uses the restricted isometry property of randomly sampled Fourier matrices, which goes back to Cand\`es and Tao \cite{CandesTao} and to Rudelson and Vershynin \cite{RudelsonVershynin}, in the form of the estimate of Burstein, Iosevich and Krause \cite[Theorem 1.6 and Corollary 5.2]{BIKRIP}. Their estimate improves that of Haviv and Regev \cite{HavivRegev} and applies directly to uniformly chosen subsets.

Let $H$ be a complex $n\times n$ matrix with entries of modulus one and $H^*H=nI$, where $H^*$ is the conjugate transpose of $H$ and $I$ is the identity matrix. The estimate of \cite{BIKRIP} states that for fixed $0<\varepsilon<\frac{1}{4}$ and $A_0>0$ there is a constant $C_{\varepsilon,A_0}$ with the following property. If $s$ is an integer with $1\leq s\leq n$, $\Omega$ is a uniformly chosen set of $m$ rows of $H$, and
\begin{equation}\label{eq:bik-condition}
C_{\varepsilon,A_0}s\log\left(\frac{2en}{s}\right)\log(2s)\leq m\leq\frac{n}{2},
\end{equation}
then, with probability at least $1-Cn^{-A_0}$, where $C$ is an absolute constant,
$$
(1-\varepsilon)\|y\|_2^2\leq\frac{1}{m}\sum_{j\in\Omega}|(Hy)_j|^2\leq(1+\varepsilon)\|y\|_2^2
$$
for every $y\in{\mathbb C}^n$ with $\|y\|_1\leq\sqrt s\|y\|_2$, in particular for every $y$ with at most $s$ nonzero entries. We only need the upper bound, with $\varepsilon=\frac{1}{8}$ and $n=N^d$, for the matrix with entries $H_{x\xi}=e^{\frac{2\pi ix\cdot\xi}{N}}$. Its rows are indexed by the points $x\in{\mathbb Z}_N^d$, and it satisfies $H^*H=N^dI$ by the orthogonality relation \eqref{eq:character-orthogonality}.

\begin{proposition}\label{prop:random-supports}
For every $A_0>0$ there are constants $c,K,C>0$, depending only on $A_0$, with the following property. Let $n=N^d$ and $K\log(2en)\leq m\leq\frac{n}{2}$, let $S$ be a uniformly chosen subset of ${\mathbb Z}_N^d$ of cardinality $m$, and let $s$ be the largest positive integer $s\leq m$ that satisfies the condition \eqref{eq:bik-condition} with $\varepsilon=\frac{1}{8}$. With probability at least $1-Cn^{-A_0}$, every function $f$ supported on $S$ satisfies
\begin{equation}\label{eq:random-restricted}
\sum_{\xi\in T}|\widehat f(\xi)|^2\leq\frac{9m}{8n}\|f\|_2^2
\end{equation}
for every set $T$ of at most $s$ frequencies, and every nonzero function $f$ supported on $S$ satisfies
\begin{equation}\label{eq:random-support-fr}
\operatorname{FR}(f)\geq\frac{2}{5}\left(\frac{ns}{m}\right)^{\frac{1}{2}}\geq\frac{cN^{\frac{d}{2}}}{\log(2eN^d)}.
\end{equation}
\end{proposition}

\begin{proof}
Take $K=C_{\frac{1}{8},A_0}\log2$, which guarantees that $s\geq1$, and work on the event of \cite{BIKRIP} with $\Omega=S$ and $\varepsilon=\frac{1}{8}$. On this event every $y$ with at most $s$ nonzero entries satisfies
$$
\sum_{x\in S}|(Hy)(x)|^2\leq\frac{9m}{8}\|y\|_2^2.
$$
Let $f$ be supported on $S$, let $T$ be a set of at most $s$ frequencies, and put $y=\widehat f1_T$. Since $f$ vanishes off $S$, the definition of the Fourier transform allows us to write
$$
\sum_{\xi\in T}|\widehat f(\xi)|^2=n^{-\frac{1}{2}}\sum_{x\in S}f(x)\overline{(Hy)(x)}.
$$
By the Cauchy--Schwarz inequality and the bound above, the right side is at most $(\frac{9m}{8n})^{\frac{1}{2}}\|f\|_2\|y\|_2$, and since $\|y\|_2^2$ is the left side, \eqref{eq:random-restricted} follows.

Now let $f\neq0$ be supported on $S$, let $b_1\geq b_2\geq\cdots\geq b_n$ be the values of $|\widehat f|$ in decreasing order, and put $\theta=\frac{9m}{8n}$, which is at most $\frac{9}{16}$ because $m\leq\frac{n}{2}$. Applied to the $s$ frequencies carrying the largest coefficients, \eqref{eq:random-restricted} tells us that $\sum_{j\leq s}b_j^2\leq\theta\|f\|_2^2$. Hence $sb_s^2\leq\theta\|f\|_2^2$ and $\sum_{j>s}b_j^2\geq(1-\theta)\|f\|_2^2\geq\frac{7}{16}\|f\|_2^2$, and in particular $b_s>0$. Since $b_j\leq b_s$ for $j>s$,
$$
\|\widehat f\|_1\geq\sum_{j>s}b_j\geq\frac{1}{b_s}\sum_{j>s}b_j^2\geq\frac{7}{16}\left(\frac{s}{\theta}\right)^{\frac{1}{2}}\|f\|_2,
$$
and because $\|\widehat f\|_2=\|f\|_2$ and $\frac{7}{16}(\frac{8}{9})^{\frac{1}{2}}>\frac{2}{5}$, this is the first inequality in \eqref{eq:random-support-fr}.

For the second inequality there is nothing to prove if $s=m$, since then $\frac{ns}{m}=n$. Otherwise $s+1$ violates the condition \eqref{eq:bik-condition}. Using $s+1\leq2s$, $\log(\frac{2en}{s+1})\leq\log(2en)$ and $\log(2s+2)\leq2\log(2s)\leq2\log(2en)$, we find that
$$
m<4C_{\frac{1}{8},A_0}s\log^2(2en).
$$
Hence $\frac{ns}{m}\geq\frac{n}{4C_{\frac{1}{8},A_0}\log^2(2en)}$, and the second inequality holds with $c=\frac{1}{5}C_{\frac{1}{8},A_0}^{-\frac{1}{2}}$.
\end{proof}

For comparison, Lemma \ref{lem:support-fr} only guarantees $\operatorname{FR}(f)\geq N^{\frac{d}{2}}m^{-\frac{1}{2}}$ for functions supported on a set of cardinality $m$, while on a random support Proposition \ref{prop:random-supports} brings every function to within a logarithm of the largest possible Fourier ratio.

Proposition \ref{prop:random-supports} applies to a broad range of cardinalities. When $|S|$ is a fixed power of $N$ smaller than $N^d$, Bourgain's selection theorem leads to a stronger conclusion. For $q>2$ we say that the characters indexed by $S$ satisfy a $\Lambda(q)$ estimate with constant $B$ if
$$
\left(\frac{1}{N^d}\sum_{\xi\in{\mathbb Z}_N^d}\left|\sum_{x\in S}a_xe^{\frac{2\pi ix\cdot\xi}{N}}\right|^q\right)^{\frac{1}{q}}\leq B\left(\sum_{x\in S}|a_x|^2\right)^{\frac{1}{2}}
$$
for every choice of complex coefficients $a_x$. Bourgain \cite{Bourgain} proved that random sets of cardinality about $N^{\frac{2d}{q}}$ satisfy such an estimate with a constant independent of $N$. Burstein, Iosevich and Krause \cite[Theorem 1.1 and Corollary 1.2]{BIKLambda} give polynomial bounds for the probability of failure, in a form that applies to uniformly chosen subsets of a prescribed cardinality and controls every coefficient vector on the same event.

\begin{corollary}\label{cor:random-supports}
Let $d\geq1$ and $0<\alpha<d$, and put $n=N^d$, $m=\lfloor N^\alpha\rfloor$ and $q=\frac{2d}{\alpha}$. There are constants $c,C>0$, depending only on $d$ and $\alpha$, such that, for all sufficiently large $N$, a uniformly chosen set $S\subseteq{\mathbb Z}_N^d$ of cardinality $m$ satisfies, with probability at least $1-n^{-1}$, both $A_{2,q}(S)\leq C$ and
$$
cN^{\frac{d}{2}}\leq\operatorname{FR}(f)\leq N^{\frac{d}{2}}
$$
for every nonzero $f$ supported on $S$. In particular, one can choose a family $\{S_N\}$ of these cardinalities such that $\kappa(\{f_N\})=0$ for every sequence of nonzero functions $f_N$ supported on $S_N$, and such that uniform synthesis on $\{S_N\}$ holds exactly for $1\leq r<\frac{2d}{\alpha}$.
\end{corollary}

\begin{proof}
The $n$ characters $\xi\mapsto e^{\frac{2\pi ix\cdot\xi}{N}}$, indexed by $x\in{\mathbb Z}_N^d$, form an orthonormal system of functions of absolute value one on ${\mathbb Z}_N^d$ with normalized counting measure. By \cite[Theorem 1.1]{BIKLambda}, a uniformly chosen set of $\lceil n^{\frac{2}{q}}\rceil=\lceil N^\alpha\rceil$ of them satisfies a $\Lambda(q)$ estimate with a constant $B_q$ independent of $N$, with probability at least $1-n^{-1}$. If this cardinality exceeds $m$, we delete one of the chosen points at random. What remains is a uniformly chosen subset of cardinality $m$, and it satisfies the same estimate, since we may set the deleted coefficient equal to zero. Thus, on an event of the stated probability,
\begin{equation}\label{eq:random-lambda}
\left(\frac{1}{n}\sum_{\xi\in{\mathbb Z}_N^d}\left|\sum_{x\in S}a_xe^{\frac{2\pi ix\cdot\xi}{N}}\right|^q\right)^{\frac{1}{q}}\leq B_q\left(\sum_{x\in S}|a_x|^2\right)^{\frac{1}{2}}
\end{equation}
for every coefficient vector on $S$.

Since $E_Sg(\xi)=\frac{1}{m}\sum_{x\in S}g(x)e^{\frac{2\pi ix\cdot\xi}{N}}$, the choice $a_x=g(x)$ in \eqref{eq:random-lambda} yields
$$
\|E_Sg\|_q\leq B_q\frac{n^{\frac{1}{q}}}{m}\|g\|_2=B_q\frac{n^{\frac{1}{q}}}{\sqrt m}\|g\|_{L^2(\mu_S)}.
$$
As $n^{\frac{2}{q}}=N^\alpha$ and $m\geq\frac{1}{2}N^\alpha$ for large $N$, the constant on the right is at most $\sqrt2B_q$, and the extension estimate follows.

For a function $f$ supported on $S$, the choice $a_x=f(x)$ in \eqref{eq:random-lambda}, with the frequency reflected, implies that
$$
\|\widehat f\|_q\leq B_qn^{\frac{1}{q}-\frac{1}{2}}\|f\|_2=B_qn^{\frac{1}{q}-\frac{1}{2}}\|\widehat f\|_2.
$$
Combining this with \eqref{eq:interpolation-fr} and dividing by $\|\widehat f\|_2$ when $f\neq0$, we get
$$
1\leq B_qn^{\frac{1}{q}-\frac{1}{2}}\operatorname{FR}(f)^{1-\frac{2}{q}},
$$
and since $q>2$ it follows that $\operatorname{FR}(f)\geq B_q^{-\frac{q}{q-2}}\sqrt n$. The upper bound is the Cauchy--Schwarz inequality, and both bounds hold for every supported function on the same event.

Now choose a set $S_N$ from this event for each sufficiently large $N$. The bounds for the Fourier ratio imply that $\kappa_N(f_N)\to0$ for every sequence of nonzero functions $f_N$ supported on $S_N$, and the cardinality estimate proves synthesis for $1\leq r<q$. To see that synthesis fails at $q$, let
$$
h_N=\frac{N^{\frac{d}{2}}}{|S_N|}1_{S_N}.
$$
Then $\widehat h_N(0)=1$ and $\widehat h_N(\xi)=E_{S_N}1(-\xi)$, so the extension estimate implies $\|\widehat h_N\|_q\leq C$. The same bound holds in $\ell^r$ for $r\geq q$, because these norms use counting measure, and synthesis fails at every such exponent. The finitely many initial moduli can be filled in arbitrarily.
\end{proof}

For the fixed powers of $N$ in Corollary \ref{cor:random-supports}, the logarithm in Proposition \ref{prop:random-supports} is therefore unnecessary. The restricted isometry estimate retains information of a different kind, because it controls individual sets of frequencies. For a set $T$ of frequencies write $\|u\|_{\ell^2(T)}=(\sum_{\xi\in T}|u(\xi)|^2)^{\frac{1}{2}}$. The estimate \eqref{eq:random-restricted} and the normalization of $E_S$ imply that
$$
\|E_Sg\|_{\ell^2(T)}\leq\left(\frac{9}{8}\right)^{\frac{1}{2}}\|g\|_{L^2(\mu_S)}
$$
for every set $T$ of at most $s$ frequencies. Both conclusions can be imposed on the same supports by intersecting the two events, whose probabilities tend to one.

For random supports, then, the extension bound at the endpoint is exactly what produces the counterexample to synthesis. The paraboloids of Section \ref{sec:arithmetic} behave differently. Their indicators also have concentration exponent zero, but along powers of a fixed odd prime and along products of all odd primes up to a growing bound their synthesis range exceeds the one obtained from cardinality. A vanishing concentration exponent does not by itself determine the synthesis range.

\section{Fourier moments and additive structure}\label{sec:structure}

The Fourier ratio involves an $\ell^1$ norm, which is often hard to compute. Higher moments are more accessible, since they can be estimated through cancellation in exponential sums or by counting additive configurations. We first see what such estimates say about the concentration exponent, and then turn to the structure of sets at the opposite extreme, where concentration is maximal.

\subsection{Generalized Salem estimates}

Fraser \cite{Fraser} measures Fourier cancellation through averages over the nonzero frequencies, and we adapt his definition to our normalization. Throughout this subsection the sets $S_N$ have cardinality exponent $\alpha$, where $0<\alpha<d$.

\begin{definition}
For $r>2$ and $0\leq s\leq\frac{1}{2}$, the family $\{S_N\}$ is uniformly $(r,s)$-Salem if
$$
\left(N^{-d}\sum_{\xi\neq0}|\widehat{1_{S_N}}(\xi)|^r\right)^{\frac{1}{r}}\lesssim N^{-\frac{d}{2}}|S_N|^{1-s}.
$$
\end{definition}

The zero frequency is left out because its value is determined by the cardinality. Larger values of $s$ mean more cancellation, and $s=\frac{1}{2}$ corresponds to square root cancellation on average.

\begin{theorem}\label{thm:salem}
If $\{S_N\}$ is uniformly $(r,s)$-Salem, then
$$
\overline\kappa(\{1_{S_N}\})\leq\min\left\{\frac{\alpha}{2},\frac{\alpha r(1-2s)}{2(r-2)}\right\}.
$$
In particular, a uniform $(r,\frac{1}{2})$-Salem estimate for any $r>2$ forces the concentration exponents at each modulus to tend to zero.
\end{theorem}

\begin{proof}
Let $F_N$ be the restriction of $\widehat{1_{S_N}}$ to the nonzero frequencies. Since $|S_N|=o(N^d)$, Plancherel tells us that
$$
\|F_N\|_2^2=|S_N|-\frac{|S_N|^2}{N^d}\asymp|S_N|,
$$
while the Salem hypothesis says that $\|F_N\|_r\lesssim N^{\frac{d}{r}-\frac{d}{2}}|S_N|^{1-s}$. The logarithmic convexity of the $\ell^r$ norms between $\ell^1$ and $\ell^r$, solved for the $\ell^1$ norm as in the proof of \eqref{eq:interpolation-fr}, yields
$$
\|F_N\|_1\geq\|F_N\|_2^{\frac{2(r-1)}{r-2}}\|F_N\|_r^{-\frac{r}{r-2}}\gtrsim N^{\frac{d}{2}}|S_N|^{\frac{rs-1}{r-2}}.
$$
Since $\|\widehat{1_{S_N}}\|_1\geq\|F_N\|_1$ and $\|\widehat{1_{S_N}}\|_2=|S_N|^{\frac{1}{2}}$, dividing by $|S_N|^{\frac{1}{2}}$ we obtain
$$
\operatorname{FR}(1_{S_N})\gtrsim N^{\frac{d}{2}}|S_N|^{\frac{r(2s-1)}{2(r-2)}}.
$$
Taking logarithms, we get an upper bound for $\kappa_N(1_{S_N})$ at every large modulus, with upper limit at most $\frac{\alpha r(1-2s)}{2(r-2)}$, and Lemma \ref{lem:support-fr} provides the bound $\frac{\alpha}{2}$.
\end{proof}

The Salem bound improves the elementary bound $\frac{\alpha}{2}$ exactly when $s>\frac{1}{r}$. Better cancellation thus forces a smaller concentration exponent, while the synthesis estimate improves as this exponent grows.

\subsection{Additive energy and maximal concentration}

For subsets $A$ and $B$ of an abelian group, the sumset is $A+B=\{a+b:a\in A\text{ and }b\in B\}$, and the doubling constant of a nonempty set $S$ is the ratio $\frac{|S+S|}{|S|}$. The additive energy of $S$ counts pairs of representations of the same sum,
$$
E(S)=\#\{(x_1,x_2,x_3,x_4)\in S^4:x_1+x_2=x_3+x_4\}.
$$
Expanding the fourth Fourier moment and using the orthogonality of characters, we find that
\begin{equation}\label{eq:energy-identity}
\sum_\xi|\widehat{1_S}(\xi)|^4=N^{-d}E(S).
\end{equation}
Indeed, the four Fourier factors contribute $N^{-2d}$, and the sum of the character over $\xi$ contributes $N^d$ exactly when the displayed additive equation holds. Also $E(S)\leq|S|^3$, since any three of the variables determine the fourth.

\begin{proposition}\label{prop:energy-algebra}
For every nonempty $S\subseteq{\mathbb Z}_N^d$,
\begin{equation}\label{eq:energy-fr}
\operatorname{FR}(1_S)\geq\frac{N^{\frac{d}{2}}|S|}{E(S)^{\frac{1}{2}}}\text{ and }E(S)\geq\frac{|S|^3}{\mathcal A_N(S)^2}.
\end{equation}
If the sets $S_N$ have cardinality exponent $\alpha$ and, for some $0\leq\eta\leq1$,
$$
E(S_N)\leq|S_N|^{3-\eta}N^{o(1)},
$$
then
$$
\overline\kappa(\{1_{S_N}\})\leq\frac{\alpha(1-\eta)}{2}.
$$
\end{proposition}

\begin{proof}
By the logarithmic convexity of the $\ell^r$ norms between $\ell^1$ and $\ell^4$,
$$
\|\widehat{1_S}\|_2\leq\|\widehat{1_S}\|_1^{\frac{1}{3}}\|\widehat{1_S}\|_4^{\frac{2}{3}}.
$$
Rearranging this inequality and using Plancherel and \eqref{eq:energy-identity}, we obtain
$$
\operatorname{FR}(1_S)\geq\frac{\|\widehat{1_S}\|_2^2}{\|\widehat{1_S}\|_4^2}=\frac{N^{\frac{d}{2}}|S|}{E(S)^{\frac{1}{2}}},
$$
and substituting \eqref{eq:algebra-kappa} we get the second inequality. Under the assumed bound for the energy, the first inequality yields
$$
\frac{N^{\frac{d}{2}}}{\operatorname{FR}(1_{S_N})}\leq|S_N|^{\frac{1-\eta}{2}}N^{o(1)},
$$
and the last assertion follows on taking logarithms and then the upper limit.
\end{proof}

For $\eta=1$, energy of order $|S_N|^{2+o(1)}$ forces the concentration exponent of the indicators to vanish. The same calculation is a generalized Salem estimate at the fourth moment, with $s=\frac{1+\eta}{4}$ and a subpolynomial loss.

The Balog--Szemer\'edi--Gowers theorem turns a lower bound for the energy into an upper bound for the doubling constant of a large subset. We use it in the following quantitative form \cite{BalogSzemeredi,Gowers}. There are absolute constants $c,C>0$ such that if a finite subset $S$ of an abelian group satisfies $E(S)\geq\frac{|S|^3}{K}$, then it contains a subset $B$ with
$$
|B|\geq cK^{-C}|S|\text{ and }|B+B|\leq CK^C|B|.
$$

\begin{corollary}\label{cor:structure}
For every nonempty $S\subseteq{\mathbb Z}_N^d$ there is a subset $B\subseteq S$ satisfying
$$
|B|\geq c\mathcal A_N(S)^{-C}|S|\text{ and }|B+B|\leq C\mathcal A_N(S)^C|B|.
$$
Consequently, if the sets $S_N$ have cardinality exponent $\alpha$ and $\kappa(\{1_{S_N}\})=\frac{\alpha}{2}$, then $E(S_N)=|S_N|^3N^{-o(1)}$, and there are subsets $B_N\subseteq S_N$ with
$$
|B_N|\geq N^{-o(1)}|S_N|\text{ and }|B_N+B_N|\leq N^{o(1)}|B_N|.
$$
\end{corollary}

\begin{proof}
By Proposition \ref{prop:energy-algebra} we may take $K=\mathcal A_N(S)^2$ in the Balog--Szemer\'edi--Gowers theorem, and absorbing the factor $2$ into the exponent $C$ we obtain the first two estimates. At maximal concentration, the equivalence \eqref{eq:maximal-algebra} tells us that $\mathcal A_N(S_N)=N^{o(1)}$, so the lower bound for the energy in \eqref{eq:energy-fr}, together with $E(S_N)\leq|S_N|^3$, proves the assertion about the energy, and the estimates for the subsets follow.
\end{proof}

The argument also works when the concentration is close to maximal. If $\kappa_N(1_{S_N})\geq\frac{\alpha}{2}-\delta+o(1)$, then $\mathcal A_N(S_N)\leq N^{\delta+o(1)}$, and the losses for the subsets above are at most $N^{C\delta+o(1)}$.

Green and Sanders \cite{GreenSanders} proved that an indicator can be written as a sum and difference of indicators of cosets, with the number of terms bounded in terms of its Fourier algebra norm, a result known as the quantitative idempotent theorem. When this norm is $N^{o(1)}$, their bound need not produce $N^{o(1)}$ terms. Corollary \ref{cor:structure} provides less structure, a large subset with small doubling instead of a decomposition into cosets, but its losses are controlled by powers of the Fourier algebra norm.

\subsection{Why small doubling does not give the converse}

Intervals have small doubling and maximal concentration, so one might hope that small doubling forces maximal concentration. Deleting a random half of an interval shows that it does not. The remaining set still has small doubling, but the random deletions push its Fourier ratio up to the largest possible size.

\begin{proposition}\label{prop:doubling-counterexample}
For every $0<\alpha<1$ there are sets $S_N\subseteq{\mathbb Z}_N$ such that $|S_N|\asymp N^\alpha$, $|S_N+S_N|\leq6|S_N|$ and $\operatorname{FR}(1_{S_N})\asymp N^{\frac{1}{2}}$. In particular, $E(S_N)\asymp|S_N|^3$ and $\kappa(\{1_{S_N}\})=0$.
\end{proposition}

\begin{proof}
Let $M=\lfloor N^\alpha\rfloor$ and $B=\{0,\ldots,M-1\}$. For $x\in B$ choose independent signs $\varepsilon_x$, equal to $1$ or $-1$ with probability $\frac{1}{2}$ each, and let $g(x)=\varepsilon_x$ on $B$ and $g=0$ elsewhere. For a fixed frequency $\xi$ the unnormalized Fourier sum is
$$
Z=\sum_{x\in B}\varepsilon_xe^{-\frac{2\pi ix\xi}{N}},
$$
and by independence ${\mathbb E}|Z|^2=M$ and ${\mathbb E}|Z|^4\leq3M^2$. To check the second bound, expand the fourth power. A term has nonzero expectation only when every sign occurs an even number of times, which happens only on the tuples described by at most three pairings of the four indices, and each pairing contributes at most $M^2$ in absolute value. H\"older's inequality on the probability space now yields
$$
{\mathbb E}|Z|\geq\frac{({\mathbb E}|Z|^2)^{\frac{3}{2}}}{({\mathbb E}|Z|^4)^{\frac{1}{2}}}\geq\sqrt{\frac{M}{3}}.
$$
Summing over the frequencies, we get ${\mathbb E}\|\widehat g\|_1\geq\sqrt{\frac{NM}{3}}$, while for every choice of signs the Cauchy--Schwarz inequality yields $\|\widehat g\|_1\leq\sqrt{NM}$. These two bounds imply that
$$
\|\widehat g\|_1\geq\frac{1}{2\sqrt3}\sqrt{NM}
$$
with probability bounded below by an absolute positive constant. Indeed, if $X=\frac{\|\widehat g\|_1}{\sqrt{NM}}$ and $a=\frac{1}{2\sqrt3}$, then $0\leq X\leq1$ and ${\mathbb E}X\geq2a$, and splitting the expectation according to whether $X<a$, we find that $2a\leq a+{\mathbb P}(X\geq a)$, that is, ${\mathbb P}(X\geq a)\geq a$.

Let $S=\{x\in B:\varepsilon_x=1\}$. Its cardinality has mean $\frac{M}{2}$ and variance $\frac{M}{4}$, so by Chebyshev's inequality $\frac{M}{3}\leq|S|\leq\frac{2M}{3}$ with probability tending to one. For large $N$ this event and the preceding one occur together for at least one choice of signs, which we fix. Since $1_S=\frac{1}{2}(1_B+g)$, Lemma \ref{lem:intervals} yields
$$
\|\widehat{1_S}\|_1\geq\frac{1}{2}\|\widehat g\|_1-\frac{1}{2}\|\widehat{1_B}\|_1\geq c\sqrt{NM}-C\sqrt N(1+\log M)\gtrsim\sqrt{NM},
$$
where the last step uses $\frac{1+\log M}{\sqrt M}\to0$. The upper bound is again the Cauchy--Schwarz inequality, and since $|S|\asymp M$, the Fourier ratio of $1_S$ is comparable to $N^{\frac{1}{2}}$.

Finally, $S+S\subseteq B+B$ and $|B+B|\leq2M$, so $|S+S|\leq6|S|$. If $r_{S+S}(z)$ is the number of representations of $z$ as a sum of two elements of $S$, then
$$
|S|^4=\left(\sum_zr_{S+S}(z)\right)^2\leq|S+S|\sum_zr_{S+S}(z)^2=|S+S|E(S),
$$
hence $E(S)\geq\frac{|S|^3}{6}$, and the trivial bound $E(S)\leq|S|^3$ completes the proof.
\end{proof}

Large additive energy, even together with bounded doubling of the whole set, therefore does not imply maximal concentration, and the converse of Corollary \ref{cor:structure} fails.

\section{Spectral synthesis and extension estimates}\label{sec:extension}

Theorem \ref{thm:affine-symmetry} tells us how small the Fourier norm of a function on $S$ can be once one of its coefficients is fixed. Extension asks the opposite question, namely how large the Fourier norm can be when the weights on the set have a prescribed norm. For even exponents the same symmetry answers both questions exactly.

Recall the extension operator $E_S$ and its norm $A_{p,q}(S)$ from the introduction. The input measure $\mu_S$ assigns mass $\frac{1}{|S|}$ to each point of $S$ and the output norm uses counting measure, the normalization of Mockenhaupt and Tao \cite{MockenhauptTao} for finite fields and of Hickman and Wright \cite{HickmanWright} for general moduli. Extending $g$ by zero off $S$, we have
$$
E_Sg(\xi)=\frac{N^{\frac{d}{2}}}{|S|}\widehat g(-\xi),
$$
so by Theorem \ref{thm:affine-symmetry}
\begin{equation}\label{eq:synthesis-indicator-extension}
\mathcal C_q(S)\|E_S1\|_q=1\text{ for }1\leq q<\infty
\end{equation}
whenever a group of invertible affine maps acts transitively on $S$. For an arbitrary nonempty set, testing the definition of $\mathcal C_q(S)$ on $1_S$ shows that the left side of \eqref{eq:synthesis-indicator-extension} is at least one. When $q$ is even, the next result replaces $\|E_S1\|_q$ by an operator norm. Its proof expands an even Fourier moment and applies H\"older's inequality with respect to a positive measure, the method that Carneiro and Oliveira e Silva \cite[Section 4.2]{CarneiroOliveiraSilva} used for Euclidean spheres.

\begin{theorem}\label{thm:extension-even}
Let $S\subseteq{\mathbb Z}_N^d$ be nonempty, and let $r=2m$, where $m$ is a positive integer. Then
\begin{equation}\label{eq:extension-infinity-even}
A_{\infty,r}(S)=\|E_S1\|_r.
\end{equation}
If a group of invertible affine maps preserves $S$ and acts transitively on it, then for every $r\leq p\leq\infty$
\begin{equation}\label{eq:exact-extension}
A_{p,r}(S)=\|E_S1\|_r=\mathcal C_r(S)^{-1}.
\end{equation}
These statements hold for every modulus $N$.
\end{theorem}

\begin{proof}
Let $\Lambda_m(S)$ be the set of tuples $(x_1,\ldots,x_m,y_1,\ldots,y_m)\in S^{2m}$ for which
$$
x_1+\cdots+x_m=y_1+\cdots+y_m
$$
in ${\mathbb Z}_N^d$. Expanding the even power and applying the orthogonality relation \eqref{eq:character-orthogonality}, we obtain
\begin{equation}\label{eq:extension-moment}
\|E_Sg\|_{2m}^{2m}=\frac{N^d}{|S|^{2m}}\sum_{\Lambda_m(S)}\prod_{j=1}^mg(x_j)\overline{g(y_j)},
\end{equation}
and in particular
\begin{equation}\label{eq:extension-indicator-moment}
\|E_S1\|_{2m}^{2m}=\frac{N^d}{|S|^{2m}}|\Lambda_m(S)|.
\end{equation}
Taking absolute values in \eqref{eq:extension-moment}, we see that
$$
\|E_Sg\|_{2m}^{2m}\leq\|E_S1\|_{2m}^{2m}\|g\|_{L^\infty(\mu_S)}^{2m},
$$
with equality for $g=1$, and this proves \eqref{eq:extension-infinity-even}.

Suppose now that the affine group is transitive. Applying the same map $T(x)=Ax+b$ to every coordinate permutes $\Lambda_m(S)$, because the linear part preserves the equality of the two sums and each sum acquires the same additional term $mb$. For each coordinate, the number of tuples in $\Lambda_m(S)$ with a prescribed value of that coordinate therefore does not depend on the value, and so it equals $\frac{|\Lambda_m(S)|}{|S|}$. We apply H\"older's inequality with $2m$ equal exponents to the sum of absolute values in \eqref{eq:extension-moment}. Each of the $2m$ factors is the $\frac{1}{2m}$ power of
$$
\frac{|\Lambda_m(S)|}{|S|}\sum_{z\in S}|g(z)|^{2m},
$$
and therefore
$$
\|E_Sg\|_{2m}^{2m}\leq\frac{N^d|\Lambda_m(S)|}{|S|^{2m}}\cdot\frac{1}{|S|}\sum_{z\in S}|g(z)|^{2m}=\|E_S1\|_{2m}^{2m}\|g\|_{L^{2m}(\mu_S)}^{2m}.
$$
Since $\mu_S$ is a probability measure, $\|g\|_{L^{2m}(\mu_S)}\leq\|g\|_{L^p(\mu_S)}$ for $p\geq2m$, and $g=1$ again attains equality. The identity \eqref{eq:synthesis-indicator-extension} completes the proof.
\end{proof}

The even exponent is used only in the identity \eqref{eq:extension-moment}, which turns a sum of products of characters into a count of solutions of an additive equation. The counting measure on these solutions is positive, and transitivity makes the number of solutions through each point the same.

\begin{corollary}\label{cor:extension-interpolation}
Suppose that $S$ satisfies the symmetry hypothesis of Theorem \ref{thm:extension-even}, and let $r\geq2$ be even. If $1<p\leq\infty$, $r\leq q<\infty$ and
$$
q\geq(r-1)p',
$$
then
\begin{equation}\label{eq:extension-interpolation}
\|E_Sg\|_q\leq\mathcal C_r(S)^{-\frac{r}{q}}\|g\|_{L^p(\mu_S)}.
\end{equation}
For every nonempty $S$ and every $1\leq p\leq\infty$, $A_{p,\infty}(S)=1$.
\end{corollary}

\begin{proof}
By the triangle inequality, $\|E_Sg\|_\infty\leq\|g\|_{L^1(\mu_S)}$, and we interpolate this estimate with the identity \eqref{eq:exact-extension} at $p=r$. With $\theta=\frac{r}{q}$ and
$$
\frac{1}{a}=1-\theta+\frac{\theta}{r}=1-\frac{r-1}{q},
$$
the Riesz--Thorin theorem yields
$$
\|E_Sg\|_q\leq\mathcal C_r(S)^{-\frac{r}{q}}\|g\|_{L^a(\mu_S)}.
$$
The hypothesis $q\geq(r-1)p'$ is equivalent to $a\leq p$, and the comparison of norms on a probability space proves \eqref{eq:extension-interpolation}. Finally, $\|E_Sg\|_\infty\leq\|g\|_{L^1(\mu_S)}\leq\|g\|_{L^p(\mu_S)}$, with equality for $g=1$ at the origin.
\end{proof}

For a transitive family and a fixed even $r$, the exact identity tells us that
\begin{equation}\label{eq:extension-uniform-equivalence}
\sup_NA_{r,r}(S_N)<\infty\text{ if and only if }\inf_N\mathcal C_r(S_N)>0,
\end{equation}
while uniform synthesis at the exponent $r$ is equivalent to
$$
A_{r,r}(S_N)\longrightarrow\infty.
$$
These statements concern different limits, because a sequence of extension constants can be unbounded without tending to infinity. The failure of a uniform extension estimate does not imply synthesis, and the failure of synthesis does not imply a uniform extension estimate.

\subsection{Exact extension constants on paraboloids}

Combined with Theorem \ref{thm:paraboloid-exact}, the extension theorem identifies the divisor sum as an extension norm.

\begin{corollary}\label{cor:paraboloid-extension}
Let $d\geq2$, let $N$ be odd, let $r\geq2$ be even, and let $r\leq p\leq\infty$. Then
\begin{equation}\label{eq:paraboloid-exact-extension}
A_{p,r}(P_N^{(d)})=D_{N,d}(r)^{\frac{1}{r}}=\left(\sum_{M\mid N}\varphi(M)M^{(d-1)(1-\frac{r}{2})}\right)^{\frac{1}{r}}.
\end{equation}
Consequently, along any family of odd moduli, a uniform $L^p(\mu_{P_N^{(d)}})$ to $\ell^r$ extension estimate holds if and only if $\sup_ND_{N,d}(r)<\infty$.
\end{corollary}

\begin{proof}
The maps \eqref{eq:paraboloid-symmetry} act transitively on $P_N^{(d)}$, so it suffices to combine Theorems \ref{thm:extension-even} and \ref{thm:paraboloid-exact}.
\end{proof}

For the parabola and $r=4$ the difference between primes and prime powers is easy to see. If $\ell$ is an odd prime, then
$$
A_{4,4}(P_\ell^{(2)})=\left(2-\frac{1}{\ell}\right)^{\frac{1}{4}},
$$
whereas
$$
A_{4,4}(P_{\ell^k}^{(2)})=\left(1+k\left(1-\frac{1}{\ell}\right)\right)^{\frac{1}{4}}.
$$
The first family therefore has a uniform $L^4$ to $\ell^4$ estimate, while for fixed $\ell$ the norm of the second is comparable to $(\log N)^{\frac{1}{4}}$, where $N=\ell^k$, the reciprocal of the endpoint synthesis constant. If $N$ is a product of distinct odd primes, then
$$
A_{4,4}(P_N^{(2)})=\left(\prod_{\ell\mid N}\left(2-\frac{1}{\ell}\right)\right)^{\frac{1}{4}},
$$
and along the moduli $N_Y$ of Corollary \ref{cor:primorial} this norm tends to infinity as the number of prime factors grows. At $r=6$ the same calculation produces
$$
A_{6,6}(P_{N_Y}^{(2)})=\left(\prod_{\substack{3\leq\ell\leq Y\\ \ell\text{ prime}}}\left(1+\frac{1}{\ell}-\frac{1}{\ell^2}\right)\right)^{\frac{1}{6}}\longrightarrow\infty,
$$
where the divergence follows from Corollary \ref{cor:primorial}.

For every input exponent and every finite output exponent $q$, testing the extension estimate on the constant function shows that
$$
D_{N,d}(q)=O(1)
$$
is necessary for a uniform extension estimate on the paraboloid. By Corollary \ref{cor:paraboloid-extension} this condition is also sufficient when $q$ is even and $p\geq q$, and Corollary \ref{cor:extension-interpolation} supplies further estimates with $p<q$. These results determine neither the full range of $p$ and $q$ nor the optimal constant at an output exponent that is not even.

\subsection{The \texorpdfstring{$L^2$}{L2} estimate and the dependence on the modulus}\label{subsec:extension-modulus}

To compare these identities with the usual $L^2$ extension problem, we must say how the constant is allowed to depend on $N$. A bound of the form
$$
A_{p,q}(P_N^{(d)})\leq C_\varepsilon N^\varepsilon\text{ for every }\varepsilon>0
$$
permits the norm to grow, as long as it grows more slowly than every positive power of $N$, and the synthesis constant can tend to zero even when its reciprocal grows this slowly. Moreover, Theorem \ref{thm:extension-even} only applies when $p\geq q$ and $q$ is even, so it does not determine the $L^2$ extension norm at exponents $q>2$.

The relevant estimates of Hickman and Wright hold in every dimension. Recall that $r_*=\frac{2(d+1)}{d-1}$. In our normalization, \cite[Theorem 6.4]{HickmanWright} states that
$$
A_{2,r_*}(P_N^{(d)})\leq C_{d,\varepsilon}N^\varepsilon
$$
for every odd $N$ and every $\varepsilon>0$, and by \cite[Theorem 6.5]{HickmanWright}
$$
A_{2,r_*}(P_{\ell^k}^{(d)})\leq C_d
$$
with $C_d$ independent of the odd prime $\ell$ and of the positive integer $k$. In dimension two the output exponent is $6$. For squarefree moduli in dimension two, Kingsbury-Neuschotz \cite[Theorem 1.1]{KingsburyNeuschotz} proved that
$$
A_{2,4}(P_N^{(2)})\leq2^{\frac{\omega(N)}{4}},
$$
where $\omega(N)$ is the number of distinct prime divisors of $N$, as in Section \ref{sec:arithmetic}. This bound is at most $C_\varepsilon N^\varepsilon$ for every $\varepsilon>0$, because $2^{\frac{1}{4}}\leq\ell^\varepsilon$ for every sufficiently large prime $\ell$ and the finitely many smaller primes contribute a constant depending only on $\varepsilon$. If $\omega(N)$ stays bounded, the extension estimate is uniform, as observed in \cite[Corollary 2.3]{KingsburyNeuschotz}. For odd squarefree moduli the exact constant has a particularly simple form.

\begin{proposition}\label{prop:squarefree-sharp-extension}
Let $N>1$ be odd and squarefree. Then
\begin{equation}\label{eq:squarefree-sharp-extension}
A_{2,4}(P_N^{(2)})=\left(\prod_{\substack{\ell\mid N\\ \ell\text{ prime}}}\left(2-\frac{1}{\ell}\right)\right)^{\frac{1}{4}}=\mathcal C_4(P_N^{(2)})^{-1},
\end{equation}
and the same value is $A_{p,4}(P_N^{(2)})$ for every $2\leq p\leq\infty$.
\end{proposition}

\begin{proof}
We identify a function on the parabola with a function of its parameter $t$ and abbreviate $E_{P_N^{(2)}}$ as $E_N$, so that
$$
E_Ng(x_1,x_2)=\frac{1}{N}\sum_{t\in{\mathbb Z}_N}g(t)e^{\frac{2\pi i}{N}(x_1t+x_2t^2)}.
$$
First let $N=\ell$ be an odd prime. Expanding the fourth power and applying the orthogonality relation \eqref{eq:character-orthogonality}, we find that
$$
\|E_\ell g\|_4^4=\frac{1}{\ell^2}\sum_{\substack{a+b=c+d\\ a^2+b^2=c^2+d^2}}g(a)g(b)\overline{g(c)g(d)},
$$
where the variables lie in ${\mathbb Z}_\ell$. Squaring the first equation and subtracting the second, we see that the two equations imply $ab=cd$, since $2$ is invertible, so the pairs $\{a,b\}$ and $\{c,d\}$ are the roots of the same quadratic polynomial over the field ${\mathbb Z}_\ell$. Hence $(c,d)$ is either $(a,b)$ or $(b,a)$, and these two choices coincide exactly when $a=b$, so that
$$
\|E_\ell g\|_4^4=\frac{1}{\ell^2}\left(2\left(\sum_t|g(t)|^2\right)^2-\sum_t|g(t)|^4\right).
$$
The Cauchy--Schwarz inequality implies that $\sum_t|g(t)|^4\geq\frac{1}{\ell}(\sum_t|g(t)|^2)^2$, and therefore
$$
\|E_\ell g\|_4^4\leq\left(2-\frac{1}{\ell}\right)\left(\frac{1}{\ell}\sum_t|g(t)|^2\right)^2.
$$
Equality holds for $g=1$, which proves the claimed constant for a prime modulus.

We next explain why the constants multiply. Let $m,n>1$ be relatively prime. The Chinese remainder theorem identifies ${\mathbb Z}_{mn}$ with ${\mathbb Z}_m\times{\mathbb Z}_n$ and the corresponding parabola with the product of the two parabolas, and it also identifies the characters, so that, up to a permutation of the output frequencies, $E_{mn}$ applies $E_m$ and $E_n$ in the two coordinates. We write $g(t,s)$ for the density in these coordinates and set $h(t,y)=E_n(g(t,\cdot))(y)$ for $t\in{\mathbb Z}_m$ and $y\in{\mathbb Z}_n^2$. Applying the estimate for $E_m$ for each fixed $y$, we obtain
$$
\|E_{mn}g\|_4\leq A_{2,4}(P_m^{(2)})\left(\sum_y\left(\frac{1}{m}\sum_t|h(t,y)|^2\right)^2\right)^{\frac{1}{4}},
$$
and the triangle inequality in $\ell^2({\mathbb Z}_n^2)$, applied to the functions $|h(t,\cdot)|^2$, shows that the right side is at most
$$
A_{2,4}(P_m^{(2)})\left(\frac{1}{m}\sum_t\left(\sum_y|h(t,y)|^4\right)^{\frac{1}{2}}\right)^{\frac{1}{2}}.
$$
Since $(\sum_y|h(t,y)|^4)^{\frac{1}{2}}=\|E_n(g(t,\cdot))\|_4^2\leq A_{2,4}(P_n^{(2)})^2\frac{1}{n}\sum_s|g(t,s)|^2$ for every $t$, we conclude that
$$
\|E_{mn}g\|_4\leq A_{2,4}(P_m^{(2)})A_{2,4}(P_n^{(2)})\left(\frac{1}{mn}\sum_{t,s}|g(t,s)|^2\right)^{\frac{1}{2}}.
$$
Iterating over the distinct prime factors proves the upper bound in \eqref{eq:squarefree-sharp-extension}. The constant density attains this bound, because its extension factors and each factor attains the bound for a prime, and the formula for $\mathcal C_4$ follows from Theorem \ref{thm:paraboloid-exact}. Finally, since the input measure is a probability measure, the same upper bound holds for $p\geq2$, and the constant density again proves equality.
\end{proof}

The formula connects extension with synthesis, and it also shows what the allowed loss conceals. Since $\ell\geq3$, every factor $2-\frac{1}{\ell}$ is at least $\frac{5}{3}$, so along odd squarefree moduli the $L^2$ to $\ell^4$ norm tends to infinity exactly when $\omega(N)$ does. By Theorem \ref{thm:paraboloid-exact} this is also the condition for synthesis at $4$, and yet the norm never exceeds $C_\varepsilon N^\varepsilon$.

Raising the output exponent slightly above $4$ does not remove this dependence on $N$. For every $1\leq p\leq\infty$ and $1\leq q<\infty$ the constant density has input norm one, so
$$
A_{p,q}(P_N^{(2)})\geq\|E_{P_N^{(2)}}1\|_q=D_{N,2}(q)^{\frac{1}{q}}=\mathcal C_q(P_N^{(2)})^{-1}.
$$
Along the products $N_Y$ of all odd primes at most $Y$, Corollary \ref{cor:primorial} shows that this lower bound tends to infinity for every $q\leq6$. Over all odd squarefree moduli, then, no estimate with a constant independent of $N$ can hold in that range, whatever the input exponent, which is consistent with the squarefree estimate at $4$ and its factor depending on the number of prime divisors.

\subsection{Higher dimensions and squarefree moduli}\label{subsec:higher-dimensions}

The synthesis calculation of Section \ref{sec:arithmetic} already covers every dimension. Over powers of a prime, the comparison with extension is just as direct, once we test the extension operator on a suitable part of the paraboloid.

\begin{proposition}\label{prop:prime-power-extension-range}
Let $d\geq2$, $1<p\leq2$ and $1\leq q\leq\infty$. A uniform $L^p(\mu_{P_{\ell^k}^{(d)}})$ to $\ell^q$ extension estimate over all odd primes $\ell$ and positive integers $k$ holds if and only if
\begin{equation}\label{eq:prime-power-extension-range}
q\geq\frac{d+1}{d-1}p'.
\end{equation}
The same condition is necessary and sufficient for estimates with a factor $C_\varepsilon N^\varepsilon$ for every $\varepsilon>0$ over all odd moduli $N$, and its necessity holds already along the powers of any fixed odd prime.
\end{proposition}

\begin{proof}
For sufficiency, let $q_0=\frac{d+1}{d-1}p'$ and $\theta=\frac{2}{p'}$, so that
$$
\frac{1}{p}=1-\theta+\frac{\theta}{2}\text{ and }\frac{1}{q_0}=\frac{\theta}{r_*}.
$$
We interpolate the uniform $L^2$ to $\ell^{r_*}$ estimate of Hickman and Wright over prime powers with the trivial $L^1$ to $\ell^\infty$ estimate. This produces the bound at $q=q_0$, uniformly over all odd prime powers, and the comparison of counting measure norms extends it to every $q\geq q_0$. The same interpolation applied to their estimate for general odd moduli yields the bound with the factor $N^\varepsilon$.

For necessity, fix an odd prime $\ell$, let $N=\ell^{2k}$, and let $H=(\ell^k{\mathbb Z}_N)^{d-1}$, a subgroup of ${\mathbb Z}_N^{d-1}$ of cardinality $N^{\frac{d-1}{2}}$. We identify a density on the paraboloid with a function of its parameter $t\in{\mathbb Z}_N^{d-1}$, write $E_N=E_{P_N^{(d)}}$, and take $g=1_H$. Every coordinate of a point $t\in H$ has square zero modulo $N$, so the points of the paraboloid above $H$ have the form $(t,0)$. The annihilator of $H$ in ${\mathbb Z}_N^{d-1}$ is $H$ itself, and character orthogonality therefore yields
$$
E_Ng(x',x_d)=N^{-\frac{d-1}{2}}1_H(x').
$$
This function is supported on $N^{\frac{d+1}{2}}$ frequencies, so for every $1<p\leq\infty$ and $1\leq q<\infty$
\begin{equation}\label{eq:prime-power-concentrated-test}
\frac{\|E_Ng\|_q}{\|g\|_{L^p(\mu_{P_N^{(d)}})}}=N^{-\frac{d-1}{2p'}+\frac{d+1}{2q}}.
\end{equation}
If \eqref{eq:prime-power-extension-range} fails, the ratio grows like a positive power of $N$, which excludes even an estimate with a factor $C_\varepsilon N^\varepsilon$ for every $\varepsilon>0$.
\end{proof}

At $p=2$ the extension exponent over prime powers is $r_*$, while synthesis along the powers of a fixed prime holds exactly up to $r_c$. The exponent $r_*$ is also the synthesis endpoint along the products $N_Y$, but for a different reason. Over prime powers it comes from the density $1_H$ in the proof above, which every extension estimate must control, whereas for the products $N_Y$ it comes from the divergence of the sum of the reciprocals of the primes.

Over a squarefree modulus, a vector whose coordinates all have square zero must vanish, so the subgroups used above are not available. The sum of the squares of the coordinates can still vanish on a nontrivial linear subspace, even over a prime modulus, and this matters for extension. In dimension three, let $\ell\equiv1\pmod4$ be prime and choose $a\in{\mathbb Z}_\ell$ with $a^2=-1$. Since $t^2+(at)^2=0$, the line
$$
L=\{(t,at,0):t\in{\mathbb Z}_\ell\}
$$
is contained in $P_\ell^{(3)}$, and character orthogonality shows that
$$
E_{P_\ell^{(3)}}1_L(x)=\frac{1}{\ell}1_{\{x_1+ax_2=0\}}.
$$
Consequently
\begin{equation}\label{eq:squarefree-line-test}
\frac{\|E_{P_\ell^{(3)}}1_L\|_q}{\|1_L\|_{L^p(\mu_{P_\ell^{(3)}})}}=\ell^{-\frac{1}{p'}+\frac{2}{q}}.
\end{equation}
Along these primes the condition $q\geq2p'$ is therefore necessary, and at $p=2$ it requires $q\geq4$, although the synthesis threshold along primes is $3$. Since primes are squarefree, the squarefree hypothesis alone cannot bring this extension exponent down to $3$.

In dimension three there is nevertheless an exact identity at $(p,q)=(2,4)$. Gonz\'alez-Riquelme and Oliveira e Silva \cite[Theorem 1.1]{GonzalezOliveiraSilva} proved that for every odd prime $\ell$ the constant density maximizes the $L^2$ to $\ell^4$ extension norm of $P_\ell^{(3)}$, whose value is therefore $(1+\frac{1}{\ell}-\frac{1}{\ell^2})^{\frac{1}{4}}$.

\begin{corollary}\label{cor:squarefree-three-dimensional}
Let $N>1$ be odd and squarefree. For every $2\leq p\leq\infty$,
\begin{equation}\label{eq:squarefree-three-dimensional}
A_{p,4}(P_N^{(3)})
=\left(\prod_{\substack{\ell\mid N\\ \ell\text{ prime}}}\left(1+\frac{1}{\ell}-\frac{1}{\ell^2}\right)\right)^{\frac{1}{4}}
=\mathcal C_4(P_N^{(3)})^{-1}.
\end{equation}
\end{corollary}

\begin{proof}
The Chinese remainder argument in the proof of Proposition \ref{prop:squarefree-sharp-extension} works equally well with vector parameters. The triangle inequality in $\ell^2$ again shows that the $L^2$ to $\ell^4$ constants multiply over distinct prime factors, and the constant density attains the product. Since the input measure is a probability measure, the same upper bound holds for $p\geq2$, and the constant density still attains it. The identity for $\mathcal C_4$ follows from Theorem \ref{thm:paraboloid-exact}.
\end{proof}

In higher dimensions, the known $L^2$ estimates over prime fields determine the squarefree range for $1<p\leq2$ once a factor $N^\varepsilon$ is allowed.

\begin{corollary}\label{cor:squarefree-extension-range}
Let $d\geq2$, and put $q_d=\frac{2(d+1)}{d-1}$ when $d$ is odd and $q_d=\frac{2(d+2)}{d}$ when $d$ is even. For $1<p\leq2$ and $1\leq q\leq\infty$, the estimate
$$
A_{p,q}(P_N^{(d)})\leq C_{d,p,q,\varepsilon}N^\varepsilon
$$
holds for every odd squarefree $N$ and every $\varepsilon>0$ if and only if
\begin{equation}\label{eq:squarefree-extension-range}
q\geq\frac{q_d}{2}p'.
\end{equation}
\end{corollary}

\begin{proof}
We first recall the estimates over prime moduli, whose constants do not depend on the odd prime $\ell$. In odd dimensions $q_d=r_*$, and the bound $A_{2,q_d}(P_\ell^{(d)})\leq C_d$ is the case $k=1$ of the theorem of Hickman and Wright quoted above. For $d=2$ the bound at $q_2=4$ is Proposition \ref{prop:squarefree-sharp-extension}. In even dimensions $d\geq4$ the estimate at $q_d$ is due to Rudnev and Shkredov \cite[Theorem 1]{RudnevShkredov} when $d=4$, to Lewko \cite[Theorem 1.2]{LewkoBilinearII} when $d=6$, and to Iosevich, Koh and Lewko \cite[Theorem 1.1]{IosevichKohLewko} when $d\geq8$.

For every $q\geq2$ the Chinese remainder argument used above shows that
$$
A_{2,q}(P_{mn}^{(d)})=A_{2,q}(P_m^{(d)})A_{2,q}(P_n^{(d)})
$$
when $m$ and $n$ are relatively prime. For the upper bound we apply the two extension estimates in turn, using the triangle inequality in $\ell^{\frac{q}{2}}$ to move the sum over the parameter outside the norm in the frequency variable, and for the lower bound we test on products of densities. Hence
$$
A_{2,q_d}(P_N^{(d)})\leq C_d^{\omega(N)}\leq C_{d,\varepsilon}N^\varepsilon,
$$
where the last inequality follows, as before, by absorbing $C_d$ into $\ell^\varepsilon$ for all sufficiently large prime divisors $\ell$. Interpolation with the $L^1$ to $\ell^\infty$ estimate and the comparison of counting measure norms now prove that the condition \eqref{eq:squarefree-extension-range} is sufficient.

For necessity, let $\ell\equiv1\pmod4$ be prime, choose $a\in{\mathbb Z}_\ell$ with $a^2=-1$, and put $w=\lfloor\frac{d-1}{2}\rfloor$. The vectors in ${\mathbb Z}_\ell^{d-1}$ whose first $2w$ coordinates come in pairs $(t_j,at_j)$, $1\leq j\leq w$, and whose remaining coordinate, if there is one, is zero, form a subspace $W$ of dimension $w$ on which the sum of the squares of the coordinates vanishes. The set $L=\{(t,0):t\in W\}$ is therefore contained in $P_\ell^{(d)}$. Its extension has magnitude $\ell^{w-(d-1)}$ on $\ell^{d-w}$ frequencies, so for finite $q$
$$
\frac{\|E_{P_\ell^{(d)}}1_L\|_q}{\|1_L\|_{L^p(\mu_{P_\ell^{(d)}})}}=\ell^{-\frac{d-1-w}{p'}+\frac{d-w}{q}}.
$$
An estimate with a factor $C_\varepsilon\ell^\varepsilon$ for every $\varepsilon>0$ forces this exponent to be nonpositive, and since $\frac{d-w}{d-1-w}=\frac{q_d}{2}$, this is the condition \eqref{eq:squarefree-extension-range}. When $d=2$ the same test works with $W=\{0\}$.
\end{proof}

The full squarefree extension problem in higher dimensions remains open. Corollary \ref{cor:squarefree-extension-range} settles the range $1<p\leq2$ when all odd squarefree moduli are allowed and a factor $N^\varepsilon$ is permitted, but the pairs with $p>2$, the exact constants, and possible improvements for prescribed families of prime divisors are not determined in general. Some of these questions are open already over prime fields. For example, in dimension three and for primes $\ell\equiv3\pmod4$ the conjectured uniform $L^2$ to $\ell^3$ estimate is still unknown, as discussed in \cite{MockenhauptTao} and in the introduction of \cite{LewkoBilinearII}. Over squarefree moduli the answer to the extension problem can therefore depend on which primes occur, while the synthesis problem is completely determined by the divisor sum.

The synthesis formula sees only the constant density, while an extension estimate must also control densities supported on the subgroups and lines above. In every dimension this difference can separate the two problems, even when their answers are expressed through the same exponents $r_c$ and $r_*$.

\subsection{Necessary conditions without symmetry}

When no transitive group is available, the Fourier ratio still provides a necessary condition for extension, which we state for general input exponents.

\begin{proposition}\label{prop:extension}
Suppose that the sets $S_N$ have cardinality exponent $\alpha$, where $0<\alpha<d$, and that $1<p\leq\infty$ and $2<q<\infty$. Assume that
\begin{equation}\label{eq:extension-assumption}
\|E_{S_N}g\|_q\leq C\|g\|_{L^p(\mu_{S_N})}
\end{equation}
holds for all $g$ and $N$, with $C$ independent of $N$. Then
$$
q\geq\frac{dp'}{\alpha}.
$$
Set $\kappa_+=\overline\kappa(\{1_{S_N}\})$. If $\kappa_+<\frac{\alpha}{2}$, then also
$$
q\geq\frac{2(d-2\kappa_+)}{\alpha-2\kappa_+},
$$
and if $\kappa_+=\frac{\alpha}{2}$, no estimate of the form \eqref{eq:extension-assumption} holds for finite $q$. At each modulus, every nonzero function $f$ supported on $S_N$ satisfies
\begin{equation}\label{eq:general-weight-obstruction}
\frac{N^{\frac{d}{2}}}{|S_N|}\frac{\|f\|_2}{\|f\|_{L^p(\mu_{S_N})}}\operatorname{FR}(f)^{-(1-\frac{2}{q})}\leq C,
\end{equation}
which for $p=2$ becomes
\begin{equation}\label{eq:hereditary-obstruction}
N^{\frac{d}{2}}|S_N|^{-\frac{1}{2}}\operatorname{FR}(f)^{-(1-\frac{2}{q})}\leq C.
\end{equation}
\end{proposition}

\begin{proof}
Choose $x_0\in S_N$ and test \eqref{eq:extension-assumption} on $g=1_{\{x_0\}}$. Its extension has constant magnitude $\frac{1}{|S_N|}$ on all $N^d$ frequencies, and $\|g\|_{L^p(\mu_{S_N})}=|S_N|^{-\frac{1}{p}}$, so
$$
N^{\frac{d}{q}}|S_N|^{-\frac{1}{p'}}\leq C.
$$
The exponent $\frac{d}{q}-\frac{\alpha}{p'}$ therefore cannot be positive, which is the first assertion, and the same test rules out every finite $q$ when $p=1$.

For general $f$, the relation between extension and the unitary Fourier transform turns \eqref{eq:extension-assumption} into
$$
\frac{N^{\frac{d}{2}}}{|S_N|}\frac{\|\widehat f\|_q}{\|f\|_{L^p(\mu_{S_N})}}\leq C,
$$
and rearranging \eqref{eq:interpolation-fr} with exponent $q$ and using Plancherel, we get
$$
\|\widehat f\|_q\geq\|f\|_2\operatorname{FR}(f)^{-(1-\frac{2}{q})}.
$$
Together these prove \eqref{eq:general-weight-obstruction}, and since $\|f\|_{L^2(\mu_{S_N})}=|S_N|^{-\frac{1}{2}}\|f\|_2$, we also obtain \eqref{eq:hereditary-obstruction}.

Now take $f=1_{S_N}$ and choose a subsequence along which $\kappa_N(1_{S_N})\to\kappa_+$. Since $\|1\|_{L^p(\mu_{S_N})}=1$, we obtain
$$
N^{\frac{d-\alpha}{2}-(\frac{d}{2}-\kappa_+)(1-\frac{2}{q})+o(1)}\leq C
$$
along this subsequence. The exponent cannot be positive, and rearranging, we find that
$$
q(\alpha-2\kappa_+)\geq2(d-2\kappa_+).
$$
Dividing by $\alpha-2\kappa_+$ yields the lower bound for $q$ when $\kappa_+<\frac{\alpha}{2}$, while at $\kappa_+=\frac{\alpha}{2}$ the inequality would say that $0\geq2(d-\alpha)$, which is impossible.
\end{proof}

Testing \eqref{eq:general-weight-obstruction} on $f=1_A$, where $\varnothing\neq A\subseteq S_N$, we obtain
$$
N^{\frac{d}{2}}|S_N|^{\frac{1}{p}-1}|A|^{\frac{1}{2}-\frac{1}{p}}\operatorname{FR}(1_A)^{-(1-\frac{2}{q})}\leq C,
$$
with $\frac{1}{p}=0$ when $p=\infty$, so a uniform extension estimate must control every subset and not only the whole support. At $\kappa_+=0$ the test on the whole set requires $q\geq\frac{2d}{\alpha}$, and additional concentration raises this lower bound, while the tests on single points and on subsets also carry information about the input exponent. For random sets of cardinality $\lfloor N^\alpha\rfloor$, Corollary \ref{cor:random-supports} shows that the condition $q\geq\frac{2d}{\alpha}$ is also sufficient at $p=2$.

For paraboloids the exact divisor sum can say more than these power bounds. At $q=r_c$ it grows logarithmically along the powers of a fixed odd prime, which rules out a uniform extension estimate at this output exponent for every input exponent, while the same growth proves synthesis at the endpoint. At $q=r_*$ the same conclusion holds along the moduli of Corollary \ref{cor:primorial}. These statements concern constants independent of $N$, and an estimate with a loss of $N^\varepsilon$ does not provide such a uniform bound.

\section{Further questions}\label{sec:questions}

For paraboloids, the divisor sum $D_{N,d}(r)$ determines the synthesis constant at every exponent and the diagonal extension norm at every even exponent. Are there equally explicit formulas for other sets with affine symmetry? The moment curve is a natural candidate. For $d\geq2$ let
$$
\mathcal M_N^{(d)}=\{(t,t^2,\ldots,t^d):t\in{\mathbb Z}_N\}.
$$
For $b\in{\mathbb Z}_N$ the binomial formula provides the affine map
$$
(T_bx)_j=b^j+\sum_{i=1}^j\binom{j}{i}b^{j-i}x_i
$$
for $1\leq j\leq d$. Its linear part is triangular with diagonal entries one, and it takes the point with parameter $t$ to the point with parameter $t+b$, so these maps form a transitive group of affine symmetries. For a positive integer $s$ let $J_{s,d}(N)$ be the number of tuples $(t_1,\ldots,t_s,u_1,\ldots,u_s)\in{\mathbb Z}_N^{2s}$ satisfying
$$
\sum_{j=1}^st_j^k\equiv\sum_{j=1}^su_j^k\pmod N
$$
for $1\leq k\leq d$. Expanding the Fourier moment, we find that
$$
\|E_{\mathcal M_N^{(d)}}1\|_{2s}^{2s}=N^{d-2s}J_{s,d}(N),
$$
and Theorems \ref{thm:affine-symmetry} and \ref{thm:extension-even} imply that for $p\geq2s$
\begin{equation}\label{eq:moment-curve-constants}
A_{p,2s}(\mathcal M_N^{(d)})=\mathcal C_{2s}(\mathcal M_N^{(d)})^{-1}=\left(N^{d-2s}J_{s,d}(N)\right)^{\frac{1}{2s}}.
\end{equation}
Hickman and Wright \cite[Section 7, especially Corollary 7.3]{HickmanWright} prove growth estimates for these moments over prime powers, and more precise estimates over composite moduli would translate into information about both constants in \eqref{eq:moment-curve-constants}. For synthesis any divergence of $N^{d-2s}J_{s,d}(N)$ suffices, however slow.

The extension theorem also leaves a question about exponents. For which transitive sets does
$$
A_{q,q}(S)=\|E_S1\|_q
$$
hold when $q$ is not even? The averaging proof of the synthesis formula works for every $1\leq q<\infty$, but the extension proof relies on an expansion with nonnegative coefficients, which is available only at even exponents. Even when $q$ is even, the computation of $A_{p,q}(S)$ for $p<q$ can require information about subsets of $S$ that the constant function does not see, as the subgroups and lines of Section \ref{sec:extension} show.

For sets with little affine symmetry, Theorem \ref{thm:intro-finite} remains available, but the indicator need not determine the best synthesis constant. It would be useful to find geometric assumptions under which $\mathcal C_r(S)$ is comparable to $\frac{N^{-\frac{d}{2}}|S|}{\|\widehat{1_S}\|_r}$ uniformly in $N$. There is also a question about near equality. If a function nearly attains the constant in Theorem \ref{thm:affine-symmetry}, how close must it be to a modulated indicator? The strict convexity argument handles the case of equality, and a quantitative version should specify the norm in which the distance is measured and its dependence on the modulus.

For random supports of cardinality $\lfloor N^\alpha\rfloor$, with $0<\alpha<d$ fixed, Corollary \ref{cor:random-supports} removes the logarithmic loss in the bound for the Fourier ratio for every choice of weights, with constants depending on $\alpha$ and $d$. How does the best lower bound change as the cardinality ranges from logarithmic size to a fixed proportion of $N^d$? One may also ask which estimates of the type \eqref{eq:random-restricted} hold for structured sets such as paraboloids. Their indicators have concentration exponent zero, but along powers $N=\ell^k$ of a fixed odd prime the paraboloid $P_N^{(d)}$ contains the subgroup $\{(t,0):t\in(\ell^{\lceil\frac{k}{2}\rceil}{\mathbb Z}_N)^{d-1}\}$, whose indicator has positive concentration exponent, so the answer must depend on the arithmetic of the modulus.

Maximal concentration is equivalent to $\mathcal A_N(S_N)=N^{o(1)}$, and it produces the subsets with small doubling of Corollary \ref{cor:structure}, while Proposition \ref{prop:doubling-counterexample} shows that the converse fails. What additional assumptions on a set with small doubling would force its Fourier algebra norm to be subpolynomial? The random deletion example shows that control of the sumset alone does not suffice.

Finally, Theorem \ref{thm:all-pair-sharpness} relies on subgroups of the sizes available over prime powers. Can one obtain endpoint counterexamples for every admissible real pair $(\alpha,\kappa)$ along prime moduli, without increasing the prescribed ambient dimension?

\section*{Statements and Declarations}

Competing interests. The authors have no relevant financial or non-financial interests to disclose.

\end{document}